\documentclass{amsart}

\usepackage{amscd,amsmath,amssymb,amsfonts,bbm, calligra, xspace}
\usepackage[T1]{fontenc}
\usepackage{lmodern}
\usepackage{mathtools}
\usepackage{enumerate, enumitem}
\usepackage{multicol}
\usepackage[all]{xy}
\usepackage{mathrsfs}
\usepackage{footmisc}
\usepackage{yhmath}
\usepackage{stackrel}

\usepackage{xcolor}
\definecolor{darkblue}{rgb}{0,0,0.8}
\definecolor{darkgreen}{rgb}{0,0.4,0}
\usepackage[colorlinks=true,linkcolor=darkblue, citecolor=darkgreen,  urlcolor=darkgreen, unicode,pdfborder={0 0 0},final]{hyperref}

\newtheorem{thm}{Theorem}[section]
\newtheorem{prop}[thm]{Proposition}

\newtheorem{lem}[thm]{Lemma}
\newtheorem{cor}[thm]{Corollary}

\theoremstyle{definition}

\theoremstyle{remark}
\newtheorem{rem}[thm]{Remark}
\newtheorem{rems}[thm]{Remarks}

\numberwithin{equation}{section}

\newcommand{\Gr}{\mathrm{Gr}}
\newcommand{\op}{\mathrm{op}}
\newcommand{\res}{\mathrm{res}}
\newcommand{\sing}{\mathrm{sing}}
\newcommand{\cl}{\mathrm{cl}}
\newcommand{\colim}{\mathrm{colim}}

\newcommand{\RR}{\mathrm{R}}
\newcommand{\PGL}{\mathrm{PGL}}
\newcommand{\BU}{\mathrm{BU}}
\newcommand{\BGL}{\mathrm{BGL}}
\newcommand{\BSL}{\mathrm{BSL}}
\newcommand{\BPGL}{\mathrm{BPGL}}
\newcommand{\BSO}{\mathrm{BSO}}
\newcommand{\GL}{\mathrm{GL}}
\newcommand{\SL}{\mathrm{SL}}
\newcommand{\SO}{\mathrm{SO}}
\newcommand{\an}{\mathrm{an}}
\newcommand{\Tor}{\mathrm{Tor}}
\newcommand{\topo}{\mathrm{top}}
\newcommand{\et}{\mathrm{\acute{e}t}}

\newcommand{\Id}{\mathrm{Id}}

\newcommand{\Ima}{\mathrm{Im}}
\newcommand{\Hom}{\mathrm{Hom}}

\newcommand{\Spec}{\mathrm{Spec}}
\newcommand{\Br}{\mathrm{Br}}
\newcommand{\wBr}{\widetilde{\mathrm{Br}}}
\newcommand{\Pic}{\mathrm{Pic}}
\newcommand{\Sq}{\mathrm{Sq}}
\newcommand{\tors}{\mathrm{tors}}
\newcommand{\per}{\mathrm{per}}
\newcommand{\ind}{\mathrm{ind}}

\newcommand{\isoto}{\myxrightarrow{\,\sim\,}}
\makeatletter
\def\myrightarrow{{\setbox\z@\hbox{$\rightarrow$}\dimen0\ht\z@\multiply\dimen0 6\divide\dimen0 10\ht\z@\dimen0\box\z@}}
\def\myrightarrowfill@{\arrowfill@\relbar\relbar\myrightarrow}
\newcommand{\myxrightarrow}[2][]{\ext@arrow 0359\myrightarrowfill@{#1}{#2}}
\makeatother

\makeatletter
\newcommand{\extp}{\@ifnextchar^\@extp{\@extp^{\,}}}
\def\@extp^#1{\mathop{\bigwedge\nolimits^{\!#1}}}
\makeatother

\def\loccit{\emph{loc}.\kern3pt \emph{cit}.{}\xspace}
\def\eg{e.g.\kern.3em}

\def\resp {\text{resp.}\kern.3em}

\newcommand{\ci}{\mathcal{C}^{\infty}}

\def\Z{\mathbb Z}
\def\C{\mathbb C}

\def\G{\mathbb G}
\def\L{\mathbb L}

\def\Q{\mathbb Q}
\def\P{\mathbb P}
\def\R{\mathbb R}

\def\bS{\mathbb S}

\def\cA{\mathcal{A}}
\def\cB{\mathcal{B}}
\def\cC{\mathcal{C}}
\def\cN{\mathcal{N}}
\def\cM{\mathcal{M}}
\def\cL{\mathcal{L}}
\def\cO{\mathcal{O}}
\def\cE{\mathcal{E}}
\def\cF{\mathcal{F}}
\def\cG{\mathcal{G}}
\def\cH{\mathcal{H}}

\def\cK{\mathcal{K}}
\def\cM{\mathcal{M}}
\def\cP{\mathcal{P}}
\def\cQ{\mathcal{Q}}

\def\cU{\mathcal{U}}

\def\cEnd{\mathcal End}

\def\kp{\mathfrak{p}}

\def\km{\mathfrak{m}}

\def\kX{\mathfrak{X}}

\begin{document}

\title[The Brauer group of a Stein algebra]
{The Brauer group of a Stein algebra}

\author{Olivier Benoist}
\address{D\'epartement de math\'ematiques et applications, \'Ecole normale sup\'erieure, CNRS,
45 rue d'Ulm, 75230 Paris Cedex 05, France}
\email{olivier.benoist@ens.fr}

\author{James Hotchkiss}
\address{Department of Mathematics, Columbia University, New York, NY 10027}
\email{james.hotchkiss@columbia.edu}

\renewcommand{\abstractname}{Abstract}
\begin{abstract}
We investigate the Brauer group of the ring $\cO(S)$ of holomorphic functions on a finite-dimensional Stein space $S$. We provide a purely topological computation of this group and deduce a comparison theorem between the \'etale cohomology of $\Spec(\cO(S))$ and the singular cohomology of $S$ in degree~$2$. Furthermore, we prove a purity theorem when $S$ is nonsingular and study the index of classes in the Brauer group of $\cO(S)$.
\end{abstract}

\maketitle

\section{Introduction}
\label{intro}

\subsection{Stein spaces and their Stein algebras}

\textit{Stein spaces} are the complex-analytic analogues of affine algebraic varieties. They can be characterized among all complex spaces $S$ by the surjectivity of the restriction map $\cO(S)\to\cO(T)$ for all discrete subsets $T\subset S$, or, equivalently, by the vanishing of $H^k(S,\cF)$ for all coherent sheaves~$\cF$ on $S$ and all $k>0$ (see \eg \cite{GRStein}). A \textit{Stein algebra} is the~$\C$\nobreakdash-algebra~$\cO(S)$ of holomorphic functions on some Stein space~$S$.

In algebraic geometry, there is a correspondence (an equivalence of categories) between affine algebraic varieties $V$ over a field $K$ and finitely generated $K$-algebras~$R$ (given by $R=\cO(V)$ and $V=\Spec(R)$). Through this correspondence, geometric properties of $V$ translate into algebraic properties of $R$ and vice versa.

A similar equivalence of categories holds in complex-analytic geometry, between Stein spaces and holomorphic maps on the one hand, and Stein algebras and~$\C$\nobreakdash-al\-ge\-bra morphisms on the other hand (combine \cite[Satz 1]{Forster} and \cite[Theorem~5]{ForsterUniqueness} for finite-dimensional Stein spaces, and see \cite[Theorem~0.1]{Steinalgebra} in general). 

A major difference with the algebraic situation is that Stein algebras are complicated from a commutative algebra point of view (for instance, they are usually not noetherian). This makes it more difficult to relate their algebraic properties to the analytic or topological properties of the associated Stein spaces. An influential work in this direction is Forster's \cite{Forster}, where the relation between coherent sheaves on a Stein space $S$ and modules over the Stein algebra $\cO(S)$ is investigated.

The aim of the present article is to further contribute to this line of research by studying the Brauer group of a Stein algebra. Our theorems concerning the purely algebraic properties of $\cO(S)$ (such as Theorems \ref{th3}, \ref{th5} and \ref{th9}) are proved by exploiting results relating the commutative algebra of $\cO(S)$ to the geometry or the topology of $S$ (such as Corollary \ref{cor1} or Theorem \ref{th7}).

\subsection{Algebraic and topological Brauer groups}

Our first result shows that the Brauer group of a Stein algebra is of topological nature.

\begin{thm}[Propositions \ref{propcompa1} and \ref{propcompa2}]
\label{th1}
The natural morphism
\begin{equation}
\label{isoBr}
\Br(\cO(S))\to \Br_{\topo}(S)
\end{equation}
is an isomorphism for any finite-dimensional Stein space $S$.
\end{thm}

Recall that the \textit{Brauer group} $\Br(R)$ of a ring $R$ was defined by Auslander and Goldman \cite{AG} to be the set of Morita equivalence classes of Azumaya algebras over $R$, endowed with the group law induced by tensor product, generalizing the definition of the Brauer group of a field $K$ as Morita equivalence classes of central simple algebras over $K$ (see \cite{Grothbrauer,CTS}). The \textit{topological Brauer group} $\Br_{\topo}(S)$ of a topological space $S$ is defined in a similar way in \cite[I, \S\S1.1--1.2]{Grothbrauer}, using Morita equivalence classes of topological Azumaya algebras on~$S$ (see \S\ref{pardefBrauer}).

To show that \eqref{isoBr} is an isomorphism, we exploit the analytic Brauer group of~$S$ based on holomorphic Azumaya algebras (see \cite[\S2]{HS}, \cite[\S 1]{Schroeranalytic} or \S\ref{pardefBrauer}). We identify the topological and the analytic Brauer groups of $S$ using Grauert's Oka principle~\cite{GrauertOka}, and the analytic Brauer group of $S$ with the algebraic Brauer group of~$\cO(S)$ by means of Forster's aforementioned results~\cite{Forster} as developed in~\cite[Proposition 2.5, Remark~2.6]{Steinsurface}. The first half of the argument appears in the literature (see \eg \cite[Corollary 4.3]{AW1} or \cite{Sridharan}); the second half is new.

The hypothesis that $S$ is finite-dimensional in Theorem \ref{th1} is essential both for the injectivity and the surjectivity of \eqref{isoBr} (see Propositions \ref{propnotinj} and  \ref{propnotsurj}).

\subsection{Cohomological interpretations of Brauer groups}

Grothen\-dieck constructed a canonical injection
\begin{equation}
\label{BrH2}
\Br(R)\hookrightarrow H^2_{\et}(\Spec(R),\G_m)
\end{equation}
of the Brauer group of a ring $R$ into the second \'etale cohomology group of $\Spec(R)$ with invertible coefficients (see \cite[I, (2.1)]{Grothbrauer}).
Gabber identified the image of~\eqref{BrH2} as the torsion subgroup of~$H^2_{\et}(\Spec(R),\G_m)$ (\cite[Chap.\,II, Theorem~1]{Gabber}, see also the more general~\cite{deJongGabber}).
This yields a cohomological interpretation for the left-hand side of \eqref{isoBr}.

On the topological side, Serre showed that the topological Brauer group $\Br_{\topo}(T)$  of a finite CW complex~$T$ is naturally isomorphic to $H^3(S,\Z)_{\tors}$ (see \cite[I, Corollaire~1.7]{Grothbrauer}). To provide a topological interpretation for the right-hand side of~\eqref{isoBr}, we extend Serre's computation to finite-dimensional CW complexes.

\begin{thm}[Theorem \ref{th2+}]
\label{th2}
Let $T$ be a finite-dimensional CW complex. Then the canonical morphism $\Br_{\topo}(T)\to  H^3(T,\Z)_{\tors}$ is an isomorphism.
\end{thm}

The finite-dimensionality assumption in Theorem \ref{th2} cannot be dispensed with (see~\cite[Corollary~5.10]{AW1} or~\cite{HornS}). Combining Theorems \ref{th1} and \ref{th2} shows that the Brauer group of a Stein algebra can be computed purely topologically.

\begin{thm}[Theorem \ref{cor1+}]
\label{cor1}
The natural morphism
$$\Br(\cO(S))\to H^3(S,\Z)_{\tors}$$
is an isomorphism for any finite-dimensional Stein space $S$.
\end{thm}

\subsection{Degree \texorpdfstring{$2$}{2} cohomology with finite coefficients}

Let us present an application of Theorems \ref{th1} and \ref{th2}.
For any Stein space $S$, any~$k\geq 0$, and any $n\geq 1$, one can consider the comparison morphism
\begin{equation}
\label{comparison}
H^k_{\et}(\Spec(\cO(S)),\Z/n)\to H^k(S,\Z/n)
\end{equation}
between the \'etale cohomology of $\Spec(\cO(S))$ and the singular cohomology of $S$. The question whether \eqref{comparison} is an isomorphism when $S$ is finite-dimensional was raised in \cite[Remark 6.7 (v)]{Steinsurface} and proved there if $k=0$ or $k=1$ (see also \cite[Theorem 1.5]{Stein} for a positive answer for all $k\geq 0$ in the easier but related setting of Stein compacta). We answer this question for $k=2$.

\begin{thm}[Theorem \ref{th6+}]
\label{th6}
The natural morphism
\begin{equation}
\label{isoH2}
H^2_{\et}(\Spec(\cO(S)),\Z/n)\to H^2(S,\Z/n)
\end{equation}
is an isomorphism for any finite-dimensional Stein space $S$ and any $n\geq 1$.
\end{thm}

To deduce Theorem \ref{th6} from Theorem \ref{th1}, we relate the left-hand side of \eqref{isoH2} with $n$-torsion classes in $\Br(\cO(S))$ and the right-hand side of \eqref{isoH2} with $n$-torsion classes in $\Br_{\topo}(S)$. For these purposes, we respectively use Gabber's theorem that the image of the injection \eqref{BrH2} is the subgroup of torsion classes and~Theorem \ref{th2}.

Theorem \ref{th6} was our original motivation to study Brauer groups of Stein algebras. We do not know a proof of it that does not rely on a geometric incarnation of degree~$2$ cohomology classes by means of Azumaya algebras.

\subsection{Degree \texorpdfstring{$2$}{2} cohomology with invertible coefficients}

One can also consider, for any Stein space $S$ and any $k\geq 0$, a comparison morphism similar to \eqref{comparison} but with invertible coefficients:
\begin{equation}
\label{comparisonGm}
H^k_{\et}(\Spec(\cO(S)),\G_m)\to H^k(S,\cO_S^{\times}).
\end{equation}
The morphism \eqref{comparisonGm} is tautologically an isomorphism for $k=0$. It is also an isomorphism for $k=1$ when $S$ is finite-dimensional (see Lemma \ref{isolinebundles}). It is therefore natural to ask whether \eqref{comparisonGm} is always an isomorphism when $S$ is finite-dimensional. Note that, making use of Kummer exact sequences, this would imply that the morphisms \eqref{comparison} themselves are isomorphisms.

When $k=2$, we show that injectivity holds, although surjectivity fails in general.

\begin{thm}[Theorem \ref{th7+} and Corollary \ref{th7++}]
\label{th7}
The natural morphism
\begin{equation}
\label{compaGm}
H^2_{\et}(\Spec(\cO(S)),\G_m)\to H^2(S,\cO_S^{\times})
\end{equation}
is injective for any finite-dimensional Stein space $S$. There exists a connected Stein manifold such that \eqref{compaGm} is not surjective.
\end{thm}

To prove the injectivity assertion of Theorem \ref{th7}, we rely again on a geometric interpretation of degree $2$ cohomology classes. The Brauer group is not sufficient for this purpose, as it can only account for torsion classes in $H^2_{\et}(\Spec(\cO(S)),\G_m)$ (and if $S$ is singular, this group may not be torsion, see Proposition \ref{propnontorsion}). To overcome this difficulty, we make use of the \textit{bigger Brauer group}~$\wBr(R)$ of a ring $R$ introduced by  Taylor and Raeburn \cite{Taylor, RT} (with corrections by Caenepeel and Grandjean \cite{CG,Caenepeel}, see also \cite{Schroer, HeinS}), based on possibly non-unital
%Caenepell uses unitary and unital with distinct meanings.
\textit{central separable $R$-algebras} (instead of only Azumaya algebras), which satisfies~$\wBr(R)\isoto H^2_{\et}(\Spec(R),\G_m)$ (see \cite[Theorem~3.4]{CG}). 
%[Caenepell, Remark 4.3.10] claims: it is not known that central separable algebras are equivalent to finitely generated ones. If I understand correctly, one does not know whether they are étale-locally trivial because [Raeburn-Taylor, Lemma 2.1] has an incorrect proof, see [Caenepell-Grandjean, beginning of §3]. (Even the proof of [Caenepell, Proposition 3.2.2] requires finite generation and this has repercussions on [Caenepell, Proposition 3.3.5].) This is why [Heinloth, Schroer] write p.1191: "By taking the existence of splittings as defining property, and not as a consequence, we avoid the technical problems discussed in [2]". All this is made only more confusing by Caenepell's insistence that not making finite generation assumptions leads to set-theoretical problems.
We follow the point of view of Heinloth and Schr\"oer \cite{HeinS} (see \S\ref{parbigger}). We refer to Proposition~\ref{propinjalgan} for a counterpart of the first half of Theorem \ref{th7} in complex algebraic geometry.

The failure of surjectivity of \eqref{compaGm} in general was already noted by Raeburn and Taylor in \cite[p.\,462]{RT} (in the related setting of Stein compacta). The point is that, when~$S$ is a manifold, the left-hand side of \eqref{compaGm} is a torsion group (by Theorem~\ref{th4} below), but the right-hand side of \eqref{compaGm} might contain nontorsion classes.
%Failure of surjectivity even in ideal situation (over excellent nonsingular Stein compacta). 
%Taylor's proof in the algebraic case relies on a choice of good étale covers (with a trace) and makes use of global properties of these étale covers (see [Caenepell, Theorem 6.3.6]). No such argument can be made for the analytic topology. On this point precisely, Taylor and Raeburn write p.446:  "the etale topology is often not rich enough to allow other than torsion elements of H^3(-,Z) to be represented by cocycles in the étale topology".
%[Taylor, bottom of p.200] discusses in detail an infinite-dimensional approach (with kind of PGL_{\infty}-torsors) to construct sheaves of central separable algebras representing a class in H^2(S,O_S^{\times}), that fails for holomorphic functions but works for continuous functions. The situation for Azumaya algebras is very different, because PGL_n-torsors, so finite-dimensional problem!
The comments in \cite[bottom of p.\,200]{Taylor}, where Taylor explains why it is a difficult task to represent classes in $H^2(S,\cO_S^{\times})$ by central separable $\cO_S$-algebras, clarify why bigger Brauer group arguments fail to show that \eqref{compaGm} is surjective.

\subsection{The meromorphic function field}

If $S$ is a Stein space, the cohomology group $H^2_{\et}(\Spec(\cO(S)),\G_m)$ is not torsion in general (see Proposition~\ref{propnontorsion}), so the injection~$\Br(\cO(S))\to H^2_{\et}(\Spec(\cO(S)),\G_m)$ (see \eqref{BrH2}) may not be surjective. We prove its surjectivity if the local rings of $S$ are factorial, \eg if~$S$ is nonsingular.

\begin{thm}[Corollary \ref{th3+}]
\label{th3}
The natural morphism
\begin{equation*}
\label{BrauertoH2}
\Br(\cO(S))\to H^2_{\et}(\Spec(\cO(S)),\G_m)
\end{equation*}
is an isomorphism for any finite-dimensional locally factorial Stein space $S$.
\end{thm}

Theorem \ref{th3} derives at once from the following purely cohomological statement, in which~$\cM(S)$ denotes the ring of meromorphic functions on a complex space $S$ (see \cite[Chap.\,6, \S3]{GRCoherent}).

\begin{thm}[Theorem \ref{th4+}]
\label{th4}
The natural morphism
\begin{equation*}
\label{inj}
H^2_{\et}(\Spec(\cO(S)),\G_m)\to H^2_{\et}(\Spec(\cM(S)),\G_m)
\end{equation*}
is injective for any finite-dimensional locally factorial Stein space $S$. 
\end{thm}

To prove Theorem \ref{th4}, we need a way to detect the vanishing of a cohomology class in~$H^2_{\et}(\Spec(\cO(S)),\G_m)$. To do so, we exploit the first half of Theorem~\ref{th7}. In addition to this, we make use of twisted sheaves, whose utilization in the study of Brauer groups was pioneered by de Jong and Lieblich (see \cite{deJongGabber, Lieblich}).

Theorem \ref{th4} is a complex-analytic analogue of an algebraic result of Grothen\-dieck stating that for any noetherian ring $R$ whose local rings have factorial strict henselizations, if $K$ is the total ring of fractions of $R$, then the natural morphism
\begin{equation}
\label{injRK}
H^2_{\et}(\Spec(R),\G_m)\to H^2_{\et}(\Spec(K),\G_m)
\end{equation}
is injective (see  \cite[II, Corollaire 1.10]{Grothbrauer}). Grothendieck's proof relies on a study of the \'etale sheaf of Cartier divisors on $\Spec(R)$, and more precisely on its identification with the \'etale sheaf of Weil divisors on $\Spec(R)$ (see \cite[II,~(3)~p.\,71]{Grothbrauer}). 
In the setting of Theorem \ref{th4}, one can define meaningful \'etale sheaves of Cartier and Weil divisors on $\Spec(\cO(S))$. However, the natural morphism from the first to the second is neither injective nor surjective in general (at exotic geometric points of $\Spec(\cO(S))$), even when $S$ is a manifold. This makes it difficult to prove Theorem~\ref{th4} along these lines.

\subsection{Purity}

Let $R$ be a regular domain with fraction field~$K$. Since \eqref{injRK} is injective, so is the natural morphism $\Br(R)\to \Br(K)$ (this result goes back to \cite[Theorem 7.2]{AG}). In addition, Auslander and Goldman asked in~\cite[p.\,389]{AG} whether one could recover $\Br(R)$ inside~$\Br(K)$ as
\begin{equation}
\label{purityintersection}
\Br(R)=\bigcap_{\kp}\Br(R_{\kp}),
\end{equation}
where $\kp$ runs over all height $1$ prime ideals of $R$. This statement, also known as Grothendieck's purity conjecture (see \cite[III, \S 6]{Grothbrauer}), was proved in full generality by \v{C}esnavi\v{c}ius  \cite[Theorem 1.2]{Cesnavicius} (many particular cases listed in \cite[p.\,1462]{Cesnavicius} were known before, \eg the case of excellent $\Q$-algebras \cite[III, Th\'eor\`eme~6.1]{Grothbrauer}). 
When $R$ is a $\Q$-algebra, one can rewrite \eqref{purityintersection} as an exact sequence
\begin{equation}
\label{algpurity}
0\to \Br(R)\to\Br(K)\xrightarrow{\res_x}\bigoplus_x H^1_{\et}(\Spec(\kappa(x)),\Q/\Z)
\end{equation}
(see \cite[(6.4.4)]{Cesnavicius}), where $x$ runs over all codimension one points of $\Spec(R)$ with residue field $\kappa(x)$ and~$\res_x:\Br(K)\to H^1_{\et}(\Spec(\kappa(x)),\Q/\Z)$ is the residue map (for which see \eg \cite[Definition 1.4.11 (ii)]{CTS}).

We establish a purity theorem  in Stein geometry, in the spirit of \eqref{algpurity}. 

\begin{thm}[Theorem \ref{th5+}]
\label{th5}
Let $S$ be a finite-dimensional Stein manifold. There is an exact sequence
\begin{equation}
\label{anpurity}
0\to \Br(\cO(S))\to\Br(\cM(S))\xrightarrow{\res_D}\prod_D H^1_{\et}(\Spec(\cM(D)),\Q/\Z),
\end{equation}
where $D$ runs over all irreducible analytic subsets of codimension $1$ of $S$.
\end{thm}

The injectivity of $\Br(\cO(S))\to\Br(\cM(S))$ follows from Theorem \ref{th4}. Our proof of the exactness of \eqref{anpurity} at $\Br(\cM(S))$ relies on the topological computation of~$\Br(\cO(S))$ provided by Theorem \ref{cor1} and on an effacement lemma for degree~$1$ singular cohomology classes of \'etale $\cO(S)$-algebras proved in \cite[Lemma 6.4]{Steinsurface}.

\subsection{The index}
\label{parintroindex}

Let $R$ be a ring. The \textit{period} $\per(\alpha)$ of $\alpha\in\Br(R)$ is its order in the torsion group $\Br(R)$. A more delicate invariant of $\alpha$ is its \textit{index} $\ind(\alpha)$ (the gcd of the degrees  of Azumaya algebras representing $\alpha$). In the last section of this article, we study the integer $\ind(\alpha)$ in the case of Stein algebras.

We first show that, in this situation, the index can be computed topologically (the index of a topological Brauer class being defined as the gcd of the degrees of topological Azumaya algebras representing it).

\begin{thm}[Theorem \ref{th10+}]
\label{th10}
Let $S$ be a finite-dimensional Stein space. For any~$\alpha\in \Br(\cO(S))$ with image $\alpha^{\topo}$ in $\Br_{\topo}(S)$, one has $\ind(\alpha)=\ind(\alpha^{\topo})$.
\end{thm}

  If $V$ is a connected smooth projective variety over $\C$ and $\alpha\in \Br(V)$ has image~$\alpha_{\C(V)}$ in $\Br(\C(V))$, then $\ind(\alpha)=\ind(\alpha_{\C(V)})$ (the argument in \cite[Proposition 6.1]{AW2} is attributed to Saltman). We prove a Stein analogue of this result.
  
\begin{thm}[Theorem \ref{th9+}]
\label{th9}
Let $S$ be a connected Stein manifold. Fix a class ${\alpha\in \Br(\cO(S))}$ and let~$\alpha_{\cM(S)}$ be its image in $\Br(\cM(S))$. Then $\ind(\alpha)=\ind(\alpha_{\cM(S)})$.
\end{thm}

Extending the argument in \cite{AW2} to the Stein context is complicated by the fact that coherent sheaves on $S$ do not necessarily admit global finite resolutions by vector bundles (only in restriction to compact subsets of~$S$). To overcome this difficulty, we exploit the fact that the index can be computed topologically (Theorem~\ref{th9}) and the \textit{twisted K-theory} cohomology theory of \cite{DK,AS} (see~\S\ref{partwistedKth}).

Let $V$ be a connected smooth projective variety of dimension $d$ over $\C$. The period-index conjecture (see \cite[p.\,389]{CTBourbaki}) asks whether $\ind(\alpha)\mid\per(\alpha)^{d-1}$ for all~$\alpha\in \Br(\C(V))$. It holds when $d=1$ because~$\Br(\C(V))=0$ by Tsen's theorem (see \cite[II.3.1, Proposition 6]{SerreCG}), when~${d=2}$ by de Jong's period-index theorem~\cite{deJong}, and it is open when $d\geq 3$. This problem provides incentives to construct Brauer classes whose period and index differ. Many strategies have been developed to do so, using ramification, degeneration, index reduction or Hodge theory arguments (see~\cite{CT} and its appendix, as well as \cite{Kresch,Hotchkiss,dJP}).

Let $S$ be an irreducible and reduced Stein space of dimension~$d$, one can still ask whether $\ind(\alpha)\mid\per(\alpha)^{d-1}$ for all $\alpha\in\Br(\cM(S))$. This holds for~$d=1$ because~$\Br(\cM(S))=0$ by a result of M.\,Artin \cite[Proposition~3.7]{Guralnick}, for~$d=2$ by the period-index theorem of \cite[Theorem~1.4]{Stein}, and it is open for $d\geq 3$. Again, this motivates the search for Brauer classes whose period and index differ. 
As in the algebraic case, one can use ramification to produce such examples.
However, classes in the image of~${\Br(\cO(S))\to\Br(\cM(S))}$ are not susceptible to these valuative methods (discrete valuations on $\cM(S)$ are centered on $\cO(S)$  by \cite[Theorem~I]{Isssa}) and we have been unable to apply to them the other methods listed above. Our last result shows that, in spite of this, the period and the index of a class in the image of ${\Br(\cO(S))\to\Br(\cM(S))}$ may not coincide.

\begin{thm}[Theorem \ref{th8+}]
\label{th8}
There exist a connected Stein manifold $S$ and a class $\alpha\in\Ima[\Br(\cO(S))\to\Br(\cM(S))]$ such that $\ind(\alpha)\neq\per(\alpha)$.
\end{thm}

The obstruction to the equality of period and index used in the proof of Theorem~\ref{th8} is of a topological nature: we exploit the topological period-index problem put forward and studied by Antieau and Williams (see \cite{AW1,AW2}). In addition, our argument relies on Theorems \ref{th10} and \ref{th9}.

\section{Preliminaries}

\subsection{Brauer groups}
\label{pardefBrauer}

Let $(X,\cO_X)$ be a locally ringed site 
%no need to assume enough points.
(in our applications, it will be the \'etale site of a scheme, or a complex space $S$ endowed with its sheaf~$\cO_S$ of holomorphic functions, or a topological space $T$ endowed with its sheaf $\cC_T$ of complex-valued continuous functions). An \textit{Azumaya algebra} of degree $n$ on $(X,\cO_X)$ is an $\cO_X$-algebra that is locally isomorphic to $M_n(\cO_X)$. Two Azumaya $\cA$ and $\cB$ on $(X,\cO_X)$ are \textit{Morita-equivalent} if there exists an algebra isomorphism~${\cA\otimes\cEnd(\cE)\isoto\cB\otimes\cEnd(\cF)}$ for some locally free~$\cO_X$\nobreakdash-mod\-ules of finite rank $\cE$ and $\cF$. The Brauer group $\Br(X,\cO_X)$ of $(X,\cO_X)$ is the group of Morita-equivalence classes of Azumaya algebras on~$(X,\cO_X)$, endowed with the group law induced by tensor product (see \cite[I, \S2]{Grothbrauer} or \cite[Chap.\,V, \S4]{Giraud}).

Isomorphism classes of degree $n$ Azumaya algebras on $(X,\cO_X)$ are naturally in bijection with isomorphism classes of $\PGL_n(\cO_X)$-torsors, and hence are classified by~$H^1(X,\PGL_n(\cO_X))$ (see \cite[I, Remarque 5.12]{Grothbrauer} or \cite[Chap.\,V, Remarque~4.5]{Giraud}). The boundary maps~$H^1(X,\PGL_n(\cO_X))\to H^2(X,\cO_X^{\times})$ of the short exact sequences 
$0\to\cO_X^{\times}\to\GL_n(\cO_X)\to\PGL_n(\cO_X)\to 0$ induce an injective morphism
\begin{equation}
\label{injBrH2}
\Br(X,\cO_X)\to H^2(X,\cO_X^{\times})_{\tors}
\end{equation}
(see \cite[I, \S2]{Grothbrauer} or \cite[Chap.\,V, \S4.4, \S4.6]{Giraud}).

If $R$ is a ring, set $\Br(R):=\Br(\Spec(R)_{\et},\cO_{\Spec(R)})$. If $S$ is a complex space, set~$\Br_{\an}(S):=\Br(S,\cO_S)$. If $T$ is a topological space, set $\Br_{\topo}(T):=\Br(T,\cC_T)$.

\subsection{The analytification functor}
\label{paranalytification}

Here are a few recollections on analytification in Stein geometry, for use in the proofs of Theorems \ref{cor1+}, \ref{th7+}, \ref{th4+}, \ref{th5+} and~\ref{th9+}.

Let $S$ be a Stein space. To an~$\cO(S)$\nobreakdash-scheme of finite presentation~$X$, Bingener associates its \textit{analytification}: a complex space $X^{\an}\to S$ over $S$ equipped with a morphism of locally ringed spaces~$i_X:X^{\an}\to X$, characterized by the property~that 
\begin{equation*}
\label{defan}
\Hom_S(S',X^{\an})\xrightarrow{i_X\circ\, -}\Hom_{\cO(S)-\textrm{locally ringed spaces}}(S',X)
\end{equation*}
is bijective for all complex spaces $S'$ over $S$ (\cite[Satz~1.1]{Bingener}, see also~\hbox{\cite[\S 4.1]{Stein}} or \cite[\S 2.2]{Steinsurface}). Concretely, if $X=\Spec(\cO(S)[x_1,\dots,x_n]/\langle f_1,\dots, f_m\rangle)$ is affine, then $X^{\an}$ is constructed as the zero locus of $f_1,\dots, f_m\in\cO(S\times\C^n)$ in $S\times \C^n$. In general, the existence of the analytification is deduced from the affine case by choosing affine charts and gluing. This construction is functorial \cite[p.\,2]{Bingener}, compatible with fiber products \hbox{\cite[p.\,3]{Bingener}}, and turns \'etale morphisms into local biholomorphisms \cite[Satz~3.1]{Bingener}.

Let $X$ be an $\cO(S)$-scheme of finite presentation and let $\cF$ be a quasi-coherent sheaf of finite presentation. Bingener defines the \textit{analytification} of $\cF$ to be the coherent sheaf $\cF^{\an}:=i_X^*\cF$ on $X^{\an}$, and shows that the analytification functor~${\cF\mapsto\cF^{\an}}$ is exact (see \cite[p.\,3]{Bingener}). 

Define a coherent sheaf $\cG$ on $S$ to be \textit{globally of finite presentation} if it admits a presentation of the form $\cO_S^{\oplus n_2}\to\cO_S^{\oplus n_1}\to\cG\to 0$ for some integers $n_1$ and~$n_2$. In \cite[Proposition 2.5]{Steinsurface}, a monoidal adjoint equivalence of categories 
\begin{equation}
\label{monoidaleq}
 \left\{  \begin{array}{l}
    \textrm{finitely presented quasi-coherent}\\ \hspace{2em}\textrm{sheaves on } \Spec(\cO(S))
  \end{array}\right\}
    \stackrel[]{}{\rightleftarrows} 
 \left\{  \begin{array}{l}
    \textrm{globally finitely presented}\\ \hspace{.7em}\textrm{coherent sheaves on }S
  \end{array}\right\}
\end{equation}
with left-to-right functor $\cF\mapsto\cF^{\an}$ and right-to-left functor $\cG\mapsto \widetilde{\cG(S)}$ was constructed. To be precise, the statement of \cite[Proposition 2.5]{Steinsurface} uses $\cO(S)$\nobreakdash-mod\-ules instead of quasi-coherent sheaves on~${\Spec(\cO(S))}$, but these two categories are equivalent in a way respecting the tensor product (see \cite[Lemmas~\href{https://stacks.math.columbia.edu/tag/01IB}{01IB} and~\href{https://stacks.math.columbia.edu/tag/01I8}{01I8}\,(1)]{SP}). In addition, it is not stated in \loccit that the left-to-right functor is Bingener's analytification functor, but this follows at once from the definitions.

Let $X$ be an $\cO(S)$-scheme of finite presentation. Analytifying \'etale $X$-schemes induces a morphism of sites $\varepsilon_X:(X^{\an})_{\cl}\to X_{\et}$ from the site of local isomorphisms of $X^{\an}$ (which is equivalent to the usual site of $X^{\an}$, see~\cite[XI, \S4.0]{SGA43}) to the small \'etale site of $X$. If $\L$ is an \'etale sheaf on $X$, we set $\L^{\an}:=\varepsilon_X^*\L$ and we view it as a sheaf on $X^\an$. Pulling back by $\varepsilon _X$ gives rise to comparison morphisms~${H^k_{\et}(X,\L)\to H^k(X^{\an},\L^{\an})}$ for $k\geq 0$ (see \cite[\S 2.2]{Steinsurface}).

\section{The Brauer group of a Stein algebra}

\subsection{Stein algebras and Stein spaces}
\label{parSteinStein}

\begin{prop}
\label{propcompa1}
If $S$ is a Stein space, the natural morphism ${\Br_{\an}(S)\to \Br_{\topo}(S)}$ is an isomorphism.
\end{prop}

\begin{proof}
Isomorphism classes of holomorphic and topological vector bundles of rank~$r$ on $S$ are in bijection with $H^1(S,\GL_r(\cO_S))$ and $H^1(S,\GL_r(\cC_S))$ respectively. In addition, isomorphism classes of holomorphic and topological Azumaya algebras of degree~$n$ on $S$ are in bijection with $H^1(S,\PGL_n(\cO_S))$ and $H^1(S,\PGL_n(\cC_S))$ respectively (see \S\ref{pardefBrauer}). The natural morphisms $H^1(S,\GL_r(\cO_S))\to H^1(S,\GL_r(\cC_S))$  and~$H^1(S,\PGL_n(\cO_S))\to H^1(S,\PGL_n(\cC_S))$ are bijective by Grauert's Oka principle (see \cite[Satz 2 p.\,267]{GrauertOka}), hence induce a bijection between isomorphism classes of holomorphic and topological Azumaya algebras on $S$ that respects Morita equivalence.  This proves the proposition.
\end{proof}

\begin{prop}
\label{propcompa2}
If $S$ is a finite-dimensional Stein space, the natural morphism $\Br(\cO(S))\to\Br_{\an}(S)$ is an isomorphism.
\end{prop}

\begin{proof}
For $r\geq 0$, the monoidal equivalences \eqref{monoidaleq} restrict to equivalences between the categories of locally free sheaves of rank $r$ on $\Spec(\cO(S))$ and locally free sheaves of $\cO_S$-modules of rank $r$ (see \cite[Remark 2.6]{Steinsurface}). It follows that~\eqref{monoidaleq} respects Morita equivalence. In addition, the characterization of Azumaya algebras~$\cA$ on~$\Spec(\cO(S))$ (\resp of holomorphic Azumaya algebras on $(S,\cO_S)$) as those locally free algebras such that the algebra morphism~${\cA\otimes\cA^{\op}\to\cEnd(\cA)}$ given by left and right multiplication is an isomorphism (see \cite[I, Th\'eor\`eme~5.1~(ii) et Remarque 5.12]{Grothbrauer}) shows that the two arrows of \eqref{monoidaleq} send Azumaya algebras to Azumaya algebras. We deduce that $\Br(\cO(S))\to\Br_{\an}(S)$ is an isomorphism.
\end{proof}

\subsection{The topological Brauer group of a finite-dimensional CW complex}

Let $T$ be a topological space. Consider the composition
\begin{equation}
\label{BrH3}
\Br_{\topo}(T)\to H^2(T,\cC_T^{\times})_{\tors}\to H^3(T,\Z)_{\tors}
\end{equation}
of \eqref{injBrH2} and of the boundary map of the short exact sequence
$$0\to\Z\xrightarrow{2\pi i}\cC_T\xrightarrow{\exp}\cC_T^{\times}\to 0.$$
It follows from \cite[I, pp.\,48-50]{Grothbrauer} or \cite[Chap.\,V, \S 4.6]{Giraud} that the composition~\eqref{BrH3} is also induced by the compositions
\begin{equation}
\label{PGLH2H3}
H^1(T,\PGL_n(\cC_T))\to H^2(T,\Z/n)\to H^3(T,\Z)[n]
\end{equation}
of the boundary maps associated with the short exact sequences 
$$0\to\Z/n\xrightarrow{1\mapsto\exp(\frac{2\pi i}{n})\Id_n}\SL_n(\cC_T)\to\PGL_n(\cC_T)\to 0\textrm{ and }0\to\Z\xrightarrow{n}\Z\to\Z/n\to 0.$$
Serre proved that the composition \eqref{BrH3} is an isomorphism if~$T$ is a finite CW complex (see \cite[I, Corollaire~1.7]{Grothbrauer}). We relax this finiteness hypothesis.

\begin{thm}
\label{th2+}
Let $T$ be a finite-dimensional CW complex. Then the canonical morphisms $\Br_{\topo}(T)\to H^2(T,\cC_T^{\times})_{\tors}\to H^3(T,\Z)_{\tors}$ are isomorphisms.
\end{thm}

\begin{proof}
As $T$ is paracompact (see \cite{Miyazaki}), the sheaf $\cC_T$ is fine. It follows that $H^2(T,\cC_T)=H^3(T,\cC_T)=0$, so $H^2(T,\cC_T^{\times})\to H^3(T,\Z)$ is an isomorphism. The morphism $\Br_{\topo}(T)\to H^2(T,\cC_T^{\times})_{\tors}$ is injective (see \cite[I, Proposition 1.4~3\textsuperscript{o}]{Grothbrauer}).

To prove the surjectivity of $\Br_{\topo}(T)\to H^3(T,\Z)_{\tors}$, fix $\alpha\in H^3(T,\Z)[n]$. Write~$\alpha=\partial(\beta)$, where~$\partial$ is the boundary map associated to the long exact sequence of cohomology of $0\to\Z\xrightarrow{n}\Z\to\Z/n\to 0$. Choose an Eilenberg--MacLane CW complex $K(\Z/n,2)$ with finitely many cells in each dimension (see \cite[\S II.6 p.\,36]{Thom}). Let~$d$ be the dimension of~$S$ and let $i:K(\Z/n,2)_{\leq d}\hookrightarrow K(\Z/n,2)$ be the inclusion of the~$d$\nobreakdash-skeleton of $K(\Z/n,2)$. Let $\beta_0\in H^2(K(\Z/n,2),\Z/n)$ be the tautological class. By cellular approximation~\cite[Theorem 4.8]{Hatcher}, one can find a continuous map $f:T\to K(\Z/n,2)_{\leq d}$ with~$f^*i^*\beta_0=\beta$. By Serre's theorem (see \cite[I, Corollaire~1.7]{Grothbrauer}), the class $i^*\partial(\beta_0)\in H^3(K(\Z/n,2)_{\leq d},\Z)_{\tors}$ comes from~$\Br_{\topo}(K(\Z/n,2)_{\leq d})$. It follows that $\alpha=f^*i^*\partial(\beta_0)$ comes from $\Br_{\topo}(T)$.
\end{proof}

\begin{rem}
Our proof of Theorem \ref{th2+} reduces the case of finite-dimensional CW complexes to the case of finite CW complexes dealt with by Serre. However, the argument of Serre appearing in \cite[I, Corollaire~1.7]{Grothbrauer} does not apply directly to infinite finite-dimensional CW complexes, for the following reason.

Set $\PGL_{\infty}(\C):=\colim_{n\geq 1}\PGL_n(\C)$, where transition maps are given by tensorization by unit matrices. A continuous map $T\to \BPGL_{\infty}(\C)$ lifts to $\BPGL_n(\C)$ for some $n$ if $T$ is a finite CW complex (Serre uses this), but not necessarily so if $T$ is only finite-dimensional. A counterexample is given by taking $T$ to be the $3$-skeleton of $\BPGL_{\infty}(\C)$. Indeed, as the canonical class in~$H^2(\BPGL_{\infty}(\C),\Q/\Z)$ is not torsion (because its restriction to $H^2(\BPGL_{n}(\C),\Q/\Z)$ has order exactly~$n$), its restriction to~$H^2(T,\Q/\Z)$ is not torsion, and hence the natural map $T\to\BPGL_{\infty}(\C)$ cannot lift to a map $T\to\BPGL_n(\C)$ for any $n$.
% This is not a contradiction with our result because, as the canonical class in $H^2(\BPGL_{\infty},\Q/\Z)$ is not torsion, it does not give rise to a class in $H^3(T,\Z)_{\tors}$. 
 \end{rem}
 
\begin{thm}
\label{cor1+}
Let $S$ be a finite-dimensional Stein space. The natural morphisms
$$\Br(\cO(S))\to \Br_{\an}(S)\to\Br_{\topo}(S)\to H^3(S,\Z)_{\tors}$$
are isomorphisms.
\end{thm}

\begin{proof}
As $S$ has the homotopy type of a finite-di\-men\-sion\-al CW complex (see \cite[Korollar]{Hamm}), this follows from Propositions \ref{propcompa1} and \ref{propcompa2} and Theorem~\ref{th2+}.
\end{proof}

\subsection{The case of infinite-dimensional Stein spaces}

If the Stein space $S$ is not assumed to be finite-dimensional, then  $\Br(\cO(S))\to\Br_{\topo}(S)$ (equivalently $\Br(\cO(S))\to\Br_{\an}(S)$ by Proposition \ref{propcompa1}) is neither injective nor surjective in general. In this paragraph, we provide such examples for the sake of completeness. 

The following theorem is due to Eliashberg \cite{Eliashberg} and Gompf \cite{Gompf}.

\begin{thm}
\label{lemStein}
Let $T$ be a countable CW complex of dimension $n$. There exists a Stein manifold~$S$ of dimension $n$ with the same homotopy type as~$T$.
%If one does not insist that $S$ has dimension n, much easier proof, but hard to locate in literature. Only need to find a smooth manifold, as can then complexify (use Grauert and Mihalache). To do so, embed a locally finite countable simplicial complex in R^{2n+1}, triangulate R^{2n+1} compatibly (reference ? [Munkres, Elementary differential topology, Problem 7.11]) and take a regular neighborhood after barycentric subdivision.
\end{thm}

\begin{proof}
When $n=1$, one can take $S$ to be a disjoint union of copies of $\C$, minus a discrete subset. Assume that $n\geq 2$. By \cite[Theorem 13]{Whitehead}, we may assume that $T$ is a locally finite countable simplicial complex of dimension $n$. The proof of~\cite[Chap.\,3, \S2, Theorem~9]{Spanier} (applied in $\R^{2n}$ instead of $\R^{2n+1}$) shows that one can immerse~$T$ linearly in $\R^{2n}$. An appropriate tubular neighborhood $S$ of~$T$ in this immersion is a parallelizable manifold of dimension~$2n$ with the homotopy type of $T$, that is the interior of a handlebody (in the sense of \cite[\S2]{Gompfsurvey}) all of whose handles have index $\leq n$. Applying \cite[Theorem 1.3.1]{Eliashberg} when~${n\geq 3}$ and \cite[Theorem~3.1]{Gompf} when $n=2$ (see also \cite[Corollary~2.2 and Theorem~2.4]{Gompfsurvey}) shows that $S$ is homeomorphic to a Stein manifold of dimension~$n$.
\end{proof}

\begin{lem}
\label{lemBorel}
There exist a countable CW complex $T$ with~$H^2(T,\Z)=0$ and a rank~$2$ complex vector bundle $\cE$ on $T$ such that the complex vector bundle~$\cEnd(\cE)$ is trivial, and such that $c_2(\cE)\in H^4(T,\Z)$ is not nilpotent in the ring $H^*(T,\Z)$.
\end{lem}

\begin{proof}
Choose $T$ to be a homotopy fiber of the adjoint map~${\BSL_2(\C)\to\BGL_3(\C)}$. Let $\cF$ be the universal vector bundle on~$\BSL_2(\C)$ and let $\cE$ be its pull-back to $T$. 
With integral coefficients, the Leray spectral sequence for the fibration sequence $\GL_3(\C)\to T\to \BSL_2(\C)$ shows that~$H^2(T,\Z)=0$. With coefficients $\Z/2$, it reads
\begin{equation}
\label{LerayBorel}
E_2^{p,q}=H^p(\BSL_2(\C),\Z/2)\otimes H^q(\GL_3(\C),\Z/2)\Rightarrow H^{p+q}(T,\Z/2).
\end{equation}
Generators of the exterior algebra $H^*(\GL_3(\C),\Z/2)$ transgress to some Chern polynomials of~$\cEnd(\cF)$ modulo $2$ (see \cite[\S19]{Borel}), which vanish. It follows that~\eqref{LerayBorel} degenerates. The pull-back map $\Z/2[c_2(\cF)]=H^*(\BSL_2(\C),\Z/2)\to H^*(T,\Z/2)$ is therefore injective. This concludes the proof of the lemma.
\end{proof}

\begin{prop}
\label{propnotinj}
There exists a Stein space $S$ such that the natural morphism $\Br(\cO(S))\to\Br_{\an}(S)$ is not injective.
\end{prop}

\begin{proof}
Let $T$ and $\cE$ be as in Lemma \ref{lemBorel}. Use Theorem \ref{lemStein} to find a Stein space~$S_i$ of the homotopy type of the $i$th skeleton of $T$. Let $S$ be the disjoint union of the~$(S_i)_{i\geq 3}$ (so $H^2(S,\Z)=0$). Let $\cE_S$ be the vector bundle on~$S$ induced by $\cE$ (endowed with a structure of holomorphic vector bundle by the Oka principle \cite[Satz II]{GrauertOka}). The vector bundle $\cEnd(\cE_S)$ is topologically trivial, hence holomorphically trivial by the Oka principle \cite[Satz I]{GrauertOka}. Using \eqref{monoidaleq}, one can find an algebra structure $\cA$ on a trivial vector bundle on $\Spec(\cO(S))$ such that $\cA^{\an}\simeq\cEnd(\cE_S)$ as algebras. 

As~$\cA^{\an}$ is Azumaya, the criterion \cite[I, Th\'eor\`eme~5.1~(ii) et Remarque 5.12]{Grothbrauer} and \eqref{monoidaleq} show that $\cA$ is Azumaya. Its class $[\cA]\in\Br(\cO(S))$ vanishes in $\Br_{\an}(S)$ because $\cA^{\an}\simeq\cEnd(\cE_S)$. 
Suppose that $[\cA]\in\Br(\cO(S))$ vanishes. The exact sequence 
$$H^1_{\et}(\Spec(\cO(S)),\GL_2)\to H^1_{\et}(\Spec(\cO(S)),\PGL_2)\to H^2_{\et}(\Spec(\cO(S)),\G_m)$$
of \cite[Chap.\,IV, Remarque 4.2.10]{Giraud} yields a vector bundle~$\cF$ on~$\Spec(\cO(S))$ with~$\cA\simeq\cEnd(\cF)$, so $\cEnd(\cE_S)\simeq\cEnd(\cF^{\an})$ as algebras. The exact sequence 
$$H^1(S,\cO_S^{\times})\to H^1(S,\GL_2(\cO_S))\to H^1(S,\PGL_2(\cO_S))$$
provides us with a holomorphic line bundle $\cL$ on $S$ with $\cE_S\simeq \cF^{\an}\otimes\cL$ (use \cite[Chap.\,III, Proposition 3.4.5 (iv)]{Giraud}). As $H^2(S,\Z)=0$, the Oka principle~\cite[Satz~I]{GrauertOka} shows that $\cL$ is trivial, so $\cE_S\simeq \cF^{\an}$. As $\cF$ is generated by finitely many sections, so is $\cF^{\an}$, and hence so is $\cE_S$. This gives rise to a short exact sequence of vector bundles on $S$ of the form $0\to \cG\to\cO_S^{\oplus N}\to \cE_S\to 0$. Taking Chern classes shows that~$(1+c_2(\cE_S))^{-1}=c(\cG)$. This contradicts the fact that $c_2(\cE_S)$ is not nilpotent.
\end{proof}

\begin{lem}
\label{lemBGL2}
There exists a countable CW complex $T$ and a topological Azumaya algebra $\cA$ of degree $2$ on $T$ whose class in $H^3(T,\Z)_{\tors}$ (see \eqref{BrH3}) is not nilpotent in the cohomology ring of $T$.
\end{lem}

\begin{proof}
Choose $T$ of the homotopy type of $\BPGL_2(\C)$ and let $\cA$ be the tautological topological Azumaya algebra of degree $2$ on $T$. As $\PGL_2(\C)$ has the homotopy type of its maximal compact subgroup $\SO_3(\R)$, one computes that 
\begin{equation}
\label{cohoBPGL2}
H^*(T,\Z/2)=H^*(\BPGL_2(\C),\Z/2)=H^*(\BSO_3(\R),\Z/2)=\Z/2[w_2,w_3],
\end{equation}
where the $w_i$ are the Stiefel--Whitney classes. The class of the Azumaya algebra~$\cA$ in $H^2(T,\Z/2)$ (see~\eqref{PGLH2H3}) is nonzero, hence equal to $w_2$. The reduction modulo~$2$ of its class in $H^3(T,\Z)_{\tors}$ is therefore equal to $\Sq^1w_2=w_3$ (by the Wu formula \cite[Problem 8-A]{MS}) which is not nilpotent by \eqref{cohoBPGL2}.
\end{proof}

\begin{prop}
\label{propnotsurj}
There exists a Stein space $S$ such that the natural morphism $\Br(\cO(S))\to\Br_{\an}(S)$ is not surjective.
\end{prop}

\begin{proof}
Let $T$ and $\cA$ be as in Lemma \ref{lemBGL2}. Fix a Stein space~$S_i$ with the homotopy type of the $i$th skeleton of $T$ (see  Theorem \ref{lemStein}). Let $S$ be the disjoint union of the~$(S_i)_{i\geq 0}$. Let $\cA_S$ be the topological Azumaya algebra on $S$ induced by~$\cA$. Use the Oka principle \cite[Satz II]{GrauertOka} to view $\cA_S$ as a holomorphic Azumaya algebra.

Assume for contradiction that there exists an Azumaya $\cB$ on $\Spec(\cO(S))$ such that $[\cA_S]=[\cB^{\an}]$ in $\Br_{\an}(S)=\Br_{\topo}(S)$ (see Proposition \ref{propcompa1}). Using the characterization \cite[I, Th\'eor\`eme 5.1(ii)]{Grothbrauer} of Azumaya algebras, a limit argument based on \cite[Lemmas~\href{https://stacks.math.columbia.edu/tag/01ZR}{01ZR} and~\href{https://stacks.math.columbia.edu/tag/0B8W}{0B8W}\,(1)]{SP} shows the existence of a finitely generated sub-$\C$-algebra $R\subset\cO(S)$ with induced morphism $\pi:\Spec(\cO(S))\to V:=\Spec(R)$ and of an Azumaya algebra $\cC$ on $V$ such that $\cB\simeq\pi^*\cC$.

Let $V^{\an}$ and $\cC^{\an}$ be the analytifications of $V$ and $\cC$ in the sense of~\cite{GAGA} or \cite[Exp.\,XII]{SGA1}. Let $f:S\to V^{\an}$ be the holomorphic map stemming from the universal property of~$V^{\an}$ (see \cite[Exp.\,XII, Th\'eor\`eme 1.1]{SGA1}). Then ${f^*\cC^{\an}\simeq\cB^{\an}}$.
Let $\gamma\in H^3(V^{\an},\Z)_{\tors}$ be the class associated to $\cC^{\an}$ (see \eqref{BrH3}). This class is nilpotent because $V^{\an}$ has the homotopy type of a finite CW complex (see \cite[Theorem  p.\,170, Remark 1.10]{Hironaka}). It follows that so is the class $f^*\gamma\in H^3(S,\Z)_{\tors}$ associated to $f^*\cC^{\an}\simeq\cB^{\an}$. As~${[\cA_S]=[\cB^{\an}]}$, this contradicts our choice of $\cA_S$.
\end{proof}

\section{Comparison theorems in degree \texorpdfstring{$2$}{2}}

\subsection{Finite coefficients}

We start with a well-known lemma.

\begin{lem}
\label{lemlinebundles}
Let $S$ be a finite-dimensional Stein space. The natural morphisms
\begin{equation}
\label{isolinebundles}
H^1_{\et}(\Spec(\cO(S)),\G_m)\to H^1(S,\cO_S^{\times})\to H^1(S,\cC_S^{\times})
\end{equation}
are isomorphisms.
\end{lem}

\begin{proof}

As $S$ is finite-dimensional, the analytification functor (see \S\ref{paranalytification}) induces a bijection between the sets of isomorphism classes of invertible sheaves on $\Spec(\cO(S))$ and~$S$ respectively (see~\cite[Proposition 2.5 and Remark 2.6]{Steinsurface}). This shows that the first arrow of \eqref{isolinebundles} is an isomorphism.

Since $H^1(S,\cO_S)=H^2(S,\cO_S)=0$ (as~$S$ is Stein) and $H^1(S,\cC_S)=H^2(S,\cC_S)=0$ (as $\cC_S$ is a fine sheaf), the exponential exact sequences $0\to\Z\xrightarrow{2\pi i}\cO_S\xrightarrow{\exp}\cO_S^{\times}\to 0$ (see \cite[Lemma p.\,142]{GRStein}) and $0\to\Z\xrightarrow{2\pi i}\cC_S\xrightarrow{\exp}\cC_S^{\times}\to 0$ show that both groups~$H^1(S,\cO_S^{\times})$ and $H^1(S,\cC_S^{\times})$ are naturally isomorphic to $H^2(S,\Z)$. We deduce that the second arrow of \eqref{isolinebundles} is an isomorphism.
\end{proof}

\begin{thm}
\label{th6+}
Let $S$ be a finite-dimensional Stein space. The natural morphisms
\begin{equation*}
H^2_{\et}(\Spec(\cO(S)),\Z/n)\to H^2(S,\Z/n)
\end{equation*}
are isomorphisms for all $n\geq 1$.
\end{thm}

\begin{proof}
Consider the Kummer short exact sequence 
\begin{equation}
\label{ses1}
0\to\Z/n\xrightarrow{1\mapsto \exp(\frac{2\pi i}{n})}\G_m\xrightarrow{n}\G_m\to 0
\end{equation}
 of \'etale sheaves on ${X:=\Spec(\cO(S))}$ and the analogous short exact sequence 
 \begin{equation}
 \label{ses2}
 0\to\Z/n\xrightarrow{1\mapsto \exp(\frac{2\pi i}{n})}\cC_S^{\times}\xrightarrow{n}\cC_S^{\times}\to 0
 \end{equation}
on $S$. The cohomology long exact sequences of \eqref{ses1} and \eqref{ses2} induce a commutative diagram with exact rows
\begin{equation}
\label{doubleKummer}
\begin{aligned}
\xymatrix@C=1.2em@R=1em{
0\ar[r]&H^1_{\et}(X,\G_m)/n\ar^{}[r]\ar[d]&H^2_{\et}(X,\Z/n)\ar^{}[r]\ar[d]&H^2_{\et}(X,\G_m)[n]\ar[d]\ar[r]&0\\
0\ar[r]&H^1(S,\cC_S^{\times})/n\ar[r]&H^2(S,\Z/n)\ar^{}[r]&H^2(S,\cC_S^{\times})[n]\ar[r]&0.
}
\end{aligned}
\end{equation}

In view of Lemma \ref{lemlinebundles}, the left vertical arrow of \eqref{doubleKummer} is an isomorphism.

The morphism \eqref{injBrH2} identifies the top right group of \eqref{doubleKummer} with $\Br(\cO(S))[n]$ by Gabber's theorem \cite[I, Corollaire~1.7]{Grothbrauer}, and the bottom right group of \eqref{doubleKummer} with $\Br_{\topo}(S)[n]$ thanks to Theorem \ref{th2+}. As these identifications are compatible, the right vertical arrow of \eqref{doubleKummer} is an isomorphism.

The theorem now follows from the five lemma applied to \eqref{doubleKummer}.
\end{proof}

\subsection{Central separable algebras}
\label{parbigger}

In this paragraph, we gather generalities on central separable algebras that are used in the proof of Theorem \ref{th7+}.

Let $X$ be a scheme. Let $\cM$ and $\cN$ be quasi-coherent $\cO_X$-modules and fix a surjective morphism~${\lambda:\cM\otimes_{\cO_X}\cN\to\cO_X}$ of~$\cO_X$-modules. The \textit{elementary algebra}~$\cM\otimes_{\cO_X}^{\lambda}\cN$ is the (possibly non-unital) quasi-coherent associative algebra whose underlying sheaf is~$\cM\otimes_{\cO_X}\cN$, endowed with the product
\begin{equation}
\label{elementary}
(m\otimes n)\star(m'\otimes n')=\lambda(m'\otimes n)\cdot (m\otimes n').
\end{equation}

Following Heinloth and Schr\"oer \cite[Definition 2.1]{HeinS}, we say that a (possibly non-unital) quasi-coherent associative $\cO_X$-algebra $\cA$ is \textit{central separable} if it is \'etale-locally isomorphic to an elementary algebra (in \loccit, this condition is only required smooth-locally, but see \cite[Lemma~\href{https://stacks.math.columbia.edu/tag/055U}{055U}]{SP}). Let $p:U\to X$ be an \'etale morphism. A \textit{splitting} of~$\cA$ over~$U$ is the data of quasi-coherent $\cO_U$-modules $\cM$ and~$\cN$, of a surjective morphism~${\lambda:\cM\otimes_{\cO_U}\cN\to\cO_U}$, and of an algebra isomorphism~$\varphi:\cM\otimes_{\cO_U}^{\lambda}\cN\isoto p^*\cA$. 

As is explained in \cite[p.\,1191]{HeinS}, associating to $p:U\to X$ the category of splittings $(\cM,\cN,\lambda,\varphi)$ of $p^*\cA$ over $U$ (and isomorphisms of splittings) gives rise to a stack $\kX_{\cA}$ on the small \'etale site of $X$. By results of Raeburn and Taylor \cite[Lemmas~2.3 and 2.4]{RT} (and since splittings exist \'etale-locally by our assumption), the stack $\kX_{\cA}$ is a $\G_m$-gerbe. We denote by $[\cA]\in H^2_{\et}(X,\G_m)$ the class of $\kX_{\cA}$.

The following theorem originates from \cite[Corollary 4.1]{RT} (with a mistake pointed out and corrected in \cite[Theorem 3.4]{CG}). In the form below, it follows from the work of Heinloth and Schr\" oer \cite{HeinS}.

\begin{thm}
\label{thBigBrauer}
Let $X$ be a quasi-compact and quasi-separated scheme. Fix a cohomology class~${\alpha\in H^2_{\et}(X,\G_m)}$. There exists a central separable algebra $\cA$ of finite presentation over $X$ with~$[\cA]=\alpha$, a surjective \'etale morphism $p:U\to X$, and a splitting~$(\cM,\cN,\lambda,\varphi)$ of $p^*\cA$ over $U$ with $\cM$ and $\cN$ of finite presentation.
\end{thm}

\begin{proof}
By a limit argument (based on \cite[Proposition~\href{https://stacks.math.columbia.edu/tag/01ZA}{01ZA}]{SP} and \cite[Theorem~\href{https://stacks.math.columbia.edu/tag/09YQ}{09YQ}]{SP} applied with $\cF_i=\G_m$), we may assume that $X$ is noetherian. It then follows from \cite[Theorem 3.1 and its proof]{HeinS} that
%Their statement only claims that classes of central separable algebras span H^2_{\et}(X,\G_m), but the proof shows that any class is represented by a central separable algebra.
there exists a central separable algebra $\cA$ on $X$ with $[\cA]=\alpha$. By \cite[Theorem 2.5]{HeinS}, the~$\cO_X$\nobreakdash-algebra~$\cA$ may be chosen coherent. In addition, in view of the proof of \cite[Theorem~2.5~(iii)$\Rightarrow$(ii)]{HeinS}, we may assume that, for some surjective \'etale morphism~${p:U\to X}$, there exists a splitting $(\cM,\cN,\lambda,\varphi)$ of $p^*\cA$ over~$U$ with $\cM$ and $\cN$ coherent.
\end{proof}

We will use analogues of the above constructions in complex-analytic geometry. Let $S$ be a complex space. Given coherent sheaves $\cM$ and $\cN$ on $S$ and a surjective morphism~${\lambda:\cM\otimes_{\cO_S}\cN\to \cO_S}$, the \textit{elementary algebra} associated to~$(\cM,\cN,\lambda)$ is the coherent $\cO_S$-algebra whose underlying sheaf is $\cM\otimes_{\cO_S}\cN$, endowed with the product \eqref{elementary}. A coherent $\cO_S$-algebra $\cA$ is \textit{central separable} if it is locally isomorphic to an elementary algebra.
A \textit{splitting} of $\cA$ over an open subset~$V\subset S$ is the data of coherent sheaves  $\cM$ and $\cN$ on $V$, of a surjective morphism~${\lambda:\cM\otimes_{\cO_V}\cN\to\cO_V}$ and of an algebra isomorphism $\varphi:\cM\otimes_{\cO_V}^{\lambda}\cN\isoto \cA|_V$.

\begin{lem}
\label{lemgerbe}
Let $S$ be a complex space. Let $\cA$ be a central separable $\cO_S$-algebra. The stack $\kX_{\cA}$ over $S$ associating to each open subset $V\subset S$ the category of splittings~$(\cM,\cN,\lambda,\varphi)$ of~$\cA|_V$ (and isomorphisms of splittings) is an~$\cO_S^{\times}$-gerbe.
\end{lem}

\begin{proof}
Splittings of $\cA$ exist locally by assumption. Two splittings of $\cA$ on~${V\subset S}$ are isomorphic at the level of stalks by \cite[Lemma 2.3]{RT}, hence are locally isomorphic by coherence (use \cite[Annex,~\S 2.4]{GRCoherent}). In addition, an automorphism of a splitting~$(\cM,\cN,\lambda,\varphi)$ of $\cA$ on $V\subset S$ is given by multiplication by a (unique) invertible function on $\cM$ and by its inverse on $\cN$ (this holds at the level of stalks by \cite[Lemma 2.4]{RT} and globalizes by coherence).
\end{proof}

If $\cA$ is a central separable algebra on a complex space $S$, we let $[\cA]\in H^2(S,\cO_S^{\times})$ denote the class of the $\cO_S^{\times}$-gerbe $\kX_{\cA}$ defined in Lemma \ref{lemgerbe}.

\subsection{Invertible coefficients}

\begin{thm}
\label{th7+}
Let $S$ be a finite-dimensional Stein space. The natural morphism
\begin{equation}
\label{compaGm+}
H^2_{\et}(\Spec(\cO(S)),\G_m)\to H^2(S,\cO_S^{\times})
\end{equation}
is injective. 
\end{thm}

\begin{proof}
Fix $\alpha\in H^2_{\et}(\Spec(\cO(S)),\G_m)$ in the kernel of \eqref{compaGm+}. Applying Theorem~\ref{thBigBrauer} yields a central separable algebra $\cA$ of finite presentation on $\Spec(\cO(S))$ with~${[\cA]=\alpha}$, a surjective \'etale morphism $p:U\to X$, and a splitting~$(\cM,\cN,\lambda,\varphi)$ of $p^*\cA$ over $U$ with $\cM$ and $\cN$ quasi-coherent of finite presentation. Fix presentations $\cO_U^{\oplus m_2}\to\cO_U^{\oplus m_1}\to\cM\to 0$ and~$\cO_U^{\oplus n_2}\to\cO_U^{\oplus n_1}\to\cN\to 0$ of $\cM$ and~$\cN$.

We now analytify all these objects (in the sense of \S\ref{paranalytification}). The analytification~$\cA^{\an}$ of $\cA$ is a coherent $\cO_S$-algebra, which is globally of finite presentation as an $\cO_S$-module. As ${p^{\an}:U^{\an}\to S}$ is a surjective local biholomorphism, the data~$(\cM^{\an},\cN^{\an},\lambda^{\an},\varphi^{\an})$ shows that $\cA^{\an}$ admits local splittings, hence is a central separable $\cO_S$-algebra. The above presentations of $\cM$ and $\cN$ induce presentations $\cO_{U^{\an}}^{\oplus m_2}\to\cO_{U^{\an}}^{\oplus m_1}\to\cM^{\an}\to 0$ and~$\cO_{U^{\an}}^{\oplus n_2}\to\cO_{U^{\an}}^{\oplus n_1}\to\cN^{\an}\to 0$ of $\cM^{\an}$ and~$\cN^{\an}$.

The class $[\cA^{\an}]\in H^2(S,\cO_S^{\times})$ of $\cA^{\an}$ is the image of $[\cA]=\alpha$ by \eqref{compaGm+}, and hence vanishes. It follows that $\cA^{\an}$ is split over $S$. Therefore, there exist coherent sheaves~$\cM_S$ and $\cN_S$ on $S$, as well as a surjective morphism $\lambda_S:\cM_S\otimes_{\cO_S}\cN_S\to\cO_S$ and an algebra isomorphism $\varphi_S:\cM_S\otimes_{\cO_S}^{\lambda_S}\cN_S\isoto\cA^{\an}$. 

By Lemma \ref{lemgerbe}, the sheaves $p^*\cM_S$ and $\cM^{\an}$ are locally isomorphic on $U^{\an}$. It follows that, for all $x\in S$, there exists a resolution $\cO_{S,x}^{\oplus m_2}\to \cO_{S,x}^{\oplus m_1}\to(\cM_{S})_x\to 0$. By Lemma \ref{lemresolutions} below, there exists a resolution $\cO_{S}^{\oplus m'_2}\to \cO_{S}^{\oplus m'_1}\to\cM_S\to 0$ for some integers $m_1'$ and $m_2'$. The exact same argument shows the existence of integers $n'_1$ and $n'_2$ and of a resolution $\cO_{S}^{\oplus n'_2}\to \cO_{S}^{\oplus n'_1}\to\cN_S\to 0$.

Applying \cite[Proposition 2.5]{Steinsurface} (see \eqref{monoidaleq}) shows the existence of quasi-coherent sheaves of finite presentation $\cP$ and $\cQ$ on $\Spec(\cO(S))$ endowed with isomorphisms $\cP^{\an}\isoto\cM_S$ and~$\cQ^{\an}\isoto\cN_S$, of a surjective morphism of sheaves ${\mu:\cP\otimes_{\cO_{\Spec(\cO(S))}}\cQ\to\cO_{\Spec(\cO(S))}}$ such that $\mu^{\an}=\lambda$, and of an algebra isomorphism ${\psi:\cP\otimes^{\mu}_{\cO_{\Spec(\cO(S))}}\cQ\isoto\cA}$ with $\psi^{\an}=\varphi$. The central separable algebra $\cA$ is therefore split, and $\alpha=[\cA]=0$ in $H^2_{\et}(\Spec(\cO(S)),\G_m)$.
\end{proof}

\begin{lem}
\label{lemresolutions}
Let $S$ be a finite-dimensional Stein space. Let $\cG$ be a coherent sheaf on $S$. Assume that there exist integers $n_1$ and $n_2$ and resolutions 
\begin{equation}
\label{resol1}
\cO_{S,x}^{\oplus n_2}\to \cO_{S,x}^{\oplus n_1}\to\cG_x\to 0
\end{equation}
for all $x\in S$. Then there exist integers $n_1'$ and $n_2'$ and a resolution
\begin{equation}
\label{resol2}
\cO_{S}^{\oplus n'_2}\to \cO_{S}^{\oplus n'_1}\to\cG\to 0.
\end{equation}
\end{lem}

\begin{proof}
In view of \eqref{resol1}, all the stalks of $\cG$ are generated by $n_1$ germs of sections. As~$\cG$ is moreover globally generated because $S$ is Stein, it follows from \cite[Theorem 1]{Kripke} that there exist an integer $n_1'$ and a surjective morphism  $\cO_{S}^{\oplus n'_1}\to\cG$. Let $\cK$ be the kernel of this morphism, so one has a short exact sequence
\begin{equation}
\label{sesres}
0\to\cK\to \cO_{S}^{\oplus n'_1}\to\cG\to 0.
\end{equation}

Fix $x\in S$. Let  $\km_x\subset\cO_{S,x}$ be the maximal ideal of $\cO_{S,x}$. Tensoring \eqref{sesres} with the~$\cO_{S,x}$\nobreakdash-module $\cO_{S,x}/\km_x$ gives rise to an exact sequence
\begin{equation}
\label{eqTor}
\Tor_1^{\cO_{S,x}}(\cG_x,\cO_{S,x}/\km_x)\to \cK_x/\km_x\cK_x\to (\cO_{S,x}/\km_x)^{\oplus n_1'}.
\end{equation}
The exact sequence \eqref{resol1} shows that $\Tor_1^{\cO_{S,x}}(\cG_x,\cO_{S,x}/\km_x)$ is an $\cO_{S,x}/\km_x$-vector space of dimension $\leq n_2$. In view of \eqref{eqTor}, the $\cO_{S,x}/\km_x$-vector space $\cK_x/\km_x\cK_x$ has dimension $\leq n_1'+n_2$. By Nakayama's lemma, we deduce that $\cK_x$ is generated by~$n_1'+n_2$ germs of sections. As $\cK$ is moreover globally generated because $S$ is Stein, another application of \cite[Theorem 1]{Kripke} provides us with an integer $n_2'$ and a surjective morphism $\cO_{S}^{\oplus n'_2}\to\cK$. This produces the desired resolution \eqref{resol2}.
\end{proof}

\begin{cor}
\label{corvanish2}
Let $S$ be a Stein space of dimension $\leq 2$. Then 
$${H^2_{\et}(\Spec(\cO(S)),\G_m)=0.}$$
\end{cor}

\begin{proof}
The exponential exact sequence ${0\to\Z\xrightarrow{2\pi i}\cO_S\xrightarrow{\exp}\cO_S^{\times}\to 0}$ and the vanishing of $H^2(S,\cO_S)$ and $H^3(S,\cO_S)$ (as $S$ is Stein) show that $H^2(S,\cO_S^{\times})$ is isomorphic to $H^3(S,\Z)$, which is zero because $S$ has the homotopy type of a CW complex of dimension~$2$ (see~\cite[Korollar]{Hamm}). We deduce from Theorem~\ref{th7+} that $H^2_{\et}(\Spec(\cO(S)),\G_m)$ vanishes.
\end{proof}

A similar albeit simpler argument yields the following algebraic analogue of Theorem \ref{th7+}. It can be viewed as a variant of \cite[Propositions 1.3 and 1.4]{Schroeranalytic} taking nontorsion classes of $H^2_{\et}(V,\G_m)$ into account.  In its statement, analytifications are meant in the sense of \cite{GAGA} (see also~\cite[Exp.\,XII, \S1]{SGA1} when $V$ is possibly not reduced or not projective).

\begin{prop}
\label{propinjalgan}
If $V$ is a proper algebraic variety over $\C$, the natural morphism 
\begin{equation}
\label{alganGm}
H^2_{\et}(V,\G_m)\to H^2(V^{\an},\cO_{V^{\an}}^{\times})
\end{equation}
is injective
\end{prop}

\begin{proof}
Fix $\alpha$ in the kernel of \eqref{alganGm}. Let $\cA$ be as in Theorem \ref{thBigBrauer}. As~$[\cA^{\an}]=0$, there exist coherent sheaves $\cM$ and $\cN$ on $V^{\an}$, a surjective morphism of coherent sheaves ${\lambda:\cM\otimes _{\cO_{V^{\an}}}\cN\to \cO_{V^{\an}}}$ and an algebra isomorphism ${\varphi:\cM\otimes _{\cO_{V^{\an}}}^{\lambda}\cN\isoto \cA^{\an}}$. By Serre's GAGA theorem (see \cite[\S 12]{GAGA} or \cite[Exp.\,XII, Th\'eor\`eme 4.4]{SGA1}), one can algebraize $(\cM,\cN,\lambda,\varphi)$. It follows that $\cA$ is split, hence that ${\alpha=0}$.
\end{proof}

\subsection{A nontorsion class in the cohomological Brauer group}

The next proposition shows that the use of Azumaya algebras (instead of more general central separable algebras) would have been insufficient to prove Theorem \ref{th7+}.

\begin{prop}
\label{propnontorsion}
There exists a connected normal Stein space $S$ of dimension $3$ such that $H^2_{\et}(\Spec(\cO(S)),\G_m)$ is not a torsion group.
\end{prop}

\begin{proof}
Let $V\subset\P^4_{\C}$ be a cubic threefold with a single node. The variety $V$ is locally factorial (see \eg \cite[Theorem]{CDG}),
%There are more general results by Cheltsov. Hard to understand who is the first to deal with cubic threefolds with one node. Maybe [Werner, Kleine Auflösungen spezieller dreidimensionaler Varietäten, Satz p.27], though no direct connections with class group are made?
but its henselization at the singular point has class group isomorphic to $\Z$. It therefore follows from \cite[II, Proposition~1.7 and (7bis) p.\,75]{Grothbrauer} that $H^2_{\et}(V,\G_m)$ contains a subgroup isomorphic to $\Z$.

Fix a nontorsion $\alpha\in H^2_{\et}(V,\G_m)$. The image~$\alpha^{\an}\in H^2(V^{\an},\cO_{V^{\an}}^{\times})$ of $\alpha$ by~\eqref{alganGm} is nontorsion by Proposition \ref{propinjalgan} (here, the analytification is meant in the sense of~\cite{GAGA} or \cite[Exp.\,XII]{SGA1}). Since $H^2(V^{\an},\cO_{V^{\an}})=H^2(V,\cO_V)=0$ (by GAGA, see \cite[\S 12, Th\'eor\`eme 1]{GAGA}), the image $\beta\in H^3(V^{\an},\Z)$ of~$\alpha^{\an}$ by the boundary map of the exponential exact sequence is nontorsion.

Let $D\subset V$ be a smooth hyperplane section, and set $U:=V\setminus D$. The complex space $S:=U^{\an}$ is connected normal of dimension $3$ (because so is $U$) and Stein because $U$ is affine. The long exact sequence of cohomology with support
$$H^3_{D^{\an}}(V^{\an},\Z)\to H^3(V^{\an},\Z)\to H^3(S,\Z)$$
and the Thom isomorphism $H^3_{D^{\an}}(V^{\an},\Z)\simeq H^1(D^{\an},\Z)=0$ show that $\beta|_{S}$ is nontorsion. A fortiori, the class $\alpha^{\an}|_S\in H^2(S,\cO_S^{\times})$ is nontorsion. 

Consider the morphism $f:\Spec(\cO(S))\to U$ induced by the analytification morphism $\cO(U)\to \cO(U^{\an})=\cO(S)$ and the class $f^*(\alpha|_U)\in H^2_{\et}(\Spec(\cO(S)),\G_m)$. As the image of $f^*(\alpha|_U)$ by the natural morphism $H^2_{\et}(\Spec(\cO(S)),\G_m)\to H^2(S,\cO_S^{\times})$ coincides with the nontorsion class $\alpha^{\an}|_S$, the class $f^*(\alpha|_U)$ itself is nontorsion. This proves the proposition.
\end{proof}

\begin{rems}
(i)
The example of Proposition \ref{propnontorsion} is optimal in several respects. Let~$S$ be a Stein space.
First, if $S$ is finite-dimensional and locally factorial (\eg a manifold), then~$H^2_{\et}(\Spec(\cO(S)),\G_m)$ is torsion by Theorem \ref{th4+} below. Second, if $S$ has dimension~$\leq 2$, then the group $H^2_{\et}(\Spec(\cO(S)),\G_m)$ vanishes by Corollary~\ref{corvanish2}.

(ii) In contrast to (i), there exists a connected normal algebraic surface $V$ over~$\C$ for which $H^2_{\et}(V,\G_m)$ is not a torsion group (see \cite[II, Remarque 1.11 b)]{Grothbrauer}; we note that the example provided there is erroneously claimed to be locally factorial).
\end{rems}

\section{Holomorphic functions versus meromorphic functions}

\subsection{Twisted sheaves}
\label{partwisted}

We recall the requisite background on twisted sheaves used in the proof of Theorem \ref{th4+}. We refer to \cite[\S 2]{Lieblichmoduli} for more details.

Let $(X,\cO_X)$ be a ringed site. Giraud has constructed a bijection between the set of equivalence classes of $\cO_X^{\times}$\nobreakdash-gerbes on~$X$ and the cohomology group $H^2(X,\cO_X^{\times})$ (see \cite[Chap.\,IV, \S3.4]{Giraud}).

Fix $\alpha\in H^2(X,\cO_X^{\times})$ and let~$\pi:\kX\to X$ be an $\cO_X^{\times}$-gerbe representing $\alpha$. Endow~$\kX$ with the structure of a site ringed by setting $\cO_{\kX}:=\pi^*\cO_X$ (see \cite[Definition~2.1.1.1]{Lieblichmoduli}). An $\alpha$-\textit{twisted sheaf} on $X$ is an~$\cO_{\kX}$\nobreakdash-module~$\cF$ such that the two natural right $\cO_{\kX}^{\times}$-actions on $\cF$ (induced by the~$\cO_{\kX}$\nobreakdash-module structure and by the action of the inertia of $\kX$) coincide (see \cite[Definition 2.1.2.2]{Lieblichmoduli} for a precise definition). 
Note that when $(X,\cO_X)$ is a scheme, the gerbe $\kX$ is an algebraic stack.

 We will make use of another equivalent point of view on twisted sheaves, due to C\u{a}ld\u{a}raru~\cite{Caldararu}. Let $U_{\bullet}\to X$ be a hypercovering and let ${a\in H^0(U_2,\cO_{U_2}^{\times})}$ be a \v{C}ech cocycle~ representing $\alpha$. We denote by~${p_0,p_1:U_1\to U_0}$ the two projections and by $q_0,q_1,q_2:U_2\to U_1$ the three projections. A C\u{a}ld\u{a}raru\nobreakdash-$\alpha$\nobreakdash-twisted sheaf is an~$\cO_{U_0}$-module~$\cG$ endowed with an isomorphism~${\psi:p_0^*\cG\isoto p_1^*\cG}$ such that~$(q_1^*\psi)^{-1}\circ q_2^*\psi\circ q_0^*\psi=a$. Lieblich has constructed an equivalence between the categories of~$\alpha$\nobreakdash-twisted and C\u{a}ld\u{a}raru\nobreakdash-$\alpha$\nobreakdash-twisted sheaves on $X$ (see \cite[Proposition~2.1.3.3]{Lieblichmoduli}; the proof appears in \cite[\S 2.1.3]{Lieblichthesis}). 

If $(X,\cO_X)$ is a scheme, Lieblich has verified that an $\alpha$-twisted sheaf on $X$ is quasi-coherent if and only if so is its associated C\u{a}ld\u{a}raru-$\alpha$\nobreakdash-twisted sheaf (see \cite[Lemma 2.2.1.4]{Lieblichmoduli}). The exact same proof shows that an $\alpha$-twisted sheaf on $X$ is of finite presentation if and only if so is its associated C\u{a}ld\u{a}raru-$\alpha$\nobreakdash-twisted sheaf.

\subsection{Injectivity for locally factorial Stein spaces}

\begin{thm}
\label{th4+}
The natural morphism
\begin{equation}
\label{H2OM}
H^2_{\et}(\Spec(\cO(S)),\G_m)\to H^2_{\et}(\Spec(\cM(S)),\G_m)
\end{equation}
is injective for any finite-dimensional locally factorial Stein space. 
\end{thm}

\begin{proof}
Fix $\alpha\in H^2_{\et}(\Spec(\cO(S)),\G_m)$ in the kernel of \eqref{H2OM}. Using \cite[Theorem~\href{https://stacks.math.columbia.edu/tag/09YQ}{09YQ}]{SP} (applied with $\cF_i=\G_m$), one can find a nonzero divisor $a\in\cO(S)$ such that $\alpha|_{\Spec(\cO(S)[\frac{1}{a}])}\in  H^2_{\et}(\Spec(\cO(S)[\frac{1}{a}]),\G_m)$ vanishes.

Let $\kX\to\Spec(\cO(S))$ be a $\G_m$-gerbe representing $\alpha$ (see \S\ref{partwisted}). Our choice of~$a$ implies that the $\G_m$-gerbe $\kX_{\Spec(\cO(S)[\frac{1}{a}])}\to\Spec(\cO(S)[\frac{1}{a}])$ is trivial. It follows that there exists an invertible $\alpha$-twisted sheaf $\cL$ on $\Spec(\cO(S)[\frac{1}{a}])$ (see \cite[Lemma~3.1.1.8]{Lieblich}). Extend $\cL$ to a quasi-coherent sheaf of finite presentation $\cF$ on~$\kX$ (combine \cite[Lemma 4.3 (C2)$\Rightarrow$(E2)]{Rydh1} and \cite[Corollary B]{Rydh2}).
%Extending quasi-coherent sheaves by pushfoward is [SP, Proposition 077A].
Replacing~$\cF$ with one summand of the decomposition \cite[Proposition 2.2.1.6]{Lieblichmoduli}, we may ensure that $\cF$ is an $\alpha$-twisted sheaf on $\Spec(\cO(S))$.

Fix an \'etale hypercovering $U_{\bullet}\to \Spec(\cO(S))$ and a cocycle ${a\in H^0(U_2,\cO_{U_2}^{\times})}$ representing $\alpha$ (these exist by a theorem of Verdier, see \cite[Proposition~\href{https://stacks.math.columbia.edu/tag/09VZ}{09VZ}]{SP}).
%need the site to have fiber products. 
Let~$(\cG,\psi)$ be the C\u{a}ld\u{a}raru-$\alpha$\nobreakdash-twisted sheaf of finite presentation associated to $\cF$ (with the notation of \S\ref{partwisted}). Analytifying this data (in the sense of \S\ref{paranalytification}) gives rise to a hypercovering~$U^{\an}_{\bullet}\to S$ in the site of local isomorphisms of $S$ (which is equivalent to the classical site of $S$, see \cite[XI, \S4.0]{SGA43}), to a cocycle~$a^{\an}\in H^0(U_2^{\an},\cO_{U_2^{\an}}^{\times})$, and to a C\u{a}ld\u{a}raru-$\alpha^{\an}$\nobreakdash-twisted coherent sheaf $(\cG^{\an},\psi^{\an})$ on $S$ (where $\alpha^{\an}$ is the image of $\alpha$ by the natural morphism~$H^2_{\et}(\Spec(\cO(S)),\G_m)\to H^2(S,\cO_S^{\times})$).

Since the restriction of $\cF$ to $\kX_{\Spec(\cO(S)[\frac{1}{a}])}$ is invertible, the restriction of $\cG$ to~$(U_0)_{\Spec(\cO(S)[\frac{1}{a}])}$ is invertible, and hence the restriction of $\cG^{\an}$ to $U_0\setminus\{a=0\}$ is invertible. As $a\in\cO(S)$ is a nonzerodivisor, the zero locus of~$a$ in~$U_0$ is nowhere dense in~$U_0$. In addition, since the projection $U_0\to S$ is a local biholomorphism, the local rings of $U_0$ are factorial. It follows that the double dual (or reflexive hull)~$\cH$ of $\cG^{\an}$ is invertible.

Let ${p_0,p_1:U_1\to U_0}$ and $q_0,q_1,q_2:U_2\to U_1$ be the projections (as in \S\ref{partwisted}), so the isomorphism~$\psi^{\an}:(p_0^{\an})^*\cG^{\an}\isoto (p_1^{\an})^*\cG^{\an}$ satisfies 
$$((q_1^{\an})^*\psi^{\an})^{-1}\circ (q_2^{\an})^*\psi^{\an}\circ (q_0^{\an})^*\psi^{\an}=a^{\an}.$$
The isomorphism $\chi:(p_0^{\an})^*\cH\isoto (p_1^{\an})^*\cH$ induced by $\psi^{\an}$ then satisfies 
\begin{equation}
\label{cocycle}
((q_1^{\an})^*\chi)^{-1}\circ (q_2^{\an})^*\chi\circ (q_0^{\an})^*\chi=a^{\an}.
\end{equation}
This means that $(\cH,\chi)$ is an invertible C\u{a}ld\u{a}raru-$\alpha^{\an}$\nobreakdash-twisted coherent sheaf on $S$.

As $\cH$ is invertible hence locally isomorphic to $\cO_{U_0}$, one can use \cite[Lemma~\href{https://stacks.math.columbia.edu/tag/01GJ}{01GJ}]{SP} to find a hypercovering $V_{\bullet}\to S$ (in the site of local isomorphisms of $S$) and a morphism of hypercoverings $\xi:V_{\bullet}\to U^{\an}_{\bullet}$ such that $\xi^*\cH\simeq\cO_{V_0}$. Using this identification, the isomorphism $\xi^*{\chi}:\cO_{V_1}\isoto\cO_{V_1}$ is given by multiplication by an element~$b\in\cO(V_1)^{\times}$. Pulling back \eqref{cocycle} by $\xi$ shows that $\xi^*(a^{\an})$ is the coboundary of~$b$. It follows that $\alpha^{\an}=[a^{\an}]=[\xi^*(a^{\an})]=0$ in $H^2(S,\cO_S^{\times})$.

Since the natural morphism ~$H^2_{\et}(\Spec(\cO(S)),\G_m)\to H^2(S,\cO_S^{\times})$ is injective by Theorem \ref{th7+}, we deduce that $\alpha=0$. The theorem is proved.
\end{proof}

\begin{cor}
\label{th3+}
The natural morphism
\begin{equation}
\label{BrtoH2}
\Br(\cO(S))\to H^2_{\et}(\Spec(\cO(S)),\G_m)
\end{equation}
is an isomorphism for any finite-dimensional locally factorial Stein space. 
\end{cor}

\begin{proof}
By Gabber's theorem \cite[Chap.\,II, Theorem 1]{Gabber}, the morphism \eqref{BrtoH2} identifies $\Br(\cO(S))$ with the torsion subgroup of $H^2_{\et}(\Spec(\cO(S)),\G_m)$. In view of Theorem~\ref{th4+}, the group $H^2_{\et}(\Spec(\cO(S)),\G_m)$ injects into the Galois cohomology group~$H^2_{\et}(\Spec(\cM(S)),\G_m)$, and hence is torsion. This proves the corollary.
\end{proof}

\begin{cor}
\label{corinjOM}
The natural morphism
\begin{equation*}
\Br(\cO(S))\to \Br(\cM(S))
\end{equation*}
is injective for any finite-dimensional locally factorial Stein space. 
\end{cor}

\begin{proof}
As the groups $\Br(\cO(S))$ and $\Br(\cM(S))$ inject into $H^2_{\et}(\Spec(\cO(S)),\G_m)$ and $H^2_{\et}(\Spec(\cM(S)),\G_m)$ respectively, the corollary follows from Theorem \ref{th4+}.
\end{proof}

\begin{cor}

\label{th7++}
There exists a connected Stein manifold such that the natural morphism $H^2_{\et}(\Spec(\cO(S)),\G_m)\to H^2(S,\cO_S^{\times})$ is not surjective.
\end{cor}

\begin{proof}
Let $S$ be a Stein manifold with the homotopy type of $\bS^3$ (apply \cite[Theorem 4.1]{Mihalache} or Theorem \ref{lemStein}, or choose $S:=\{\sum_{i=1}^4z_i^2=1\}\subset\C^4$). The group~$H^2_{\et}(\Spec(\cO(S)),\G_m)$ is torsion by Theorem \ref{th4+}. The exponential exact sequence ${0\to\Z\xrightarrow{2\pi i}\cO_S\xrightarrow{\exp}\cO_S^{\times}\to 0}$ and the vanishing of $H^2(S,\cO_S)$ and~$H^3(S,\cO_S)$ (as $S$ is Stein) show that $H^2(S,\cO_S^{\times})\isoto H^3(S,\Z)\simeq\Z$ is not torsion. The morphism $H^2_{\et}(\Spec(\cO(S)),\G_m)\to H^2(S,\cO_S^{\times})$ can therefore not be surjective.
\end{proof}

\subsection{Purity for Stein manifolds}
Let $S$ be a normal Stein space and let~${D\subset S}$ be an irreducible analytic subset of codimension $1$ in $S$. Let $T\subset S$ be the connected component of $S$ containing $D$. Consider the discrete valuation on the field $\cM(T)$ given by the order of vanishing along $D$. Its valuation ring, the subring $\cO(S)_D$ of $\cM(T)$ of those meromorphic functions whose polar set (in the sense of \cite[Chap.~6, \S 3.2]{GRCoherent}) does not contain~$D$, has residue field $\cM(D)$. We denote by 
$$\res_D:\Br(\cM(S))\to H^1_{\et}(\cM(D),\Q/\Z)$$
the composition of the restriction morphism $\Br(\cM(S))\to\Br(\cM(T))$ and of the residue map
$\Br(\cM(T))\to H^1_{\et}(\cM(D),\Q/\Z)$ of \cite[Definition 1.4.11 (ii)]{CTS}.

\begin{thm}
\label{th5+}
Let $S$ be a finite-dimensional Stein manifold. 
The exact sequence
\begin{equation*}
0\to \Br(\cO(S))\to\Br(\cM(S))\xrightarrow{\res_D}\prod_D H^1_{\et}(\Spec(\cM(D)),\Q/\Z),
\end{equation*}
where $D$ runs over all irreducible analytic subsets of codimension $1$ of $S$, is exact.
\end{thm}

\begin{proof}
The injectivity of $\Br(\cO(S))\to\Br(\cM(S))$ follows from Corollary \ref{corinjOM}.

Fix~${\alpha\in\Br(\cM(S))}$ in the kernel of all the $\res_D$. Let $n$ be the order of~$\alpha$ in the torsion group $\Br(\cM(S))$. As $\cM(S)$ is a product of (possibly countably many) fields, it is an absolutely flat ring by \cite[Lemma~\href{https://stacks.math.columbia.edu/tag/092G}{092G}]{SP}. It follows that all quasi-compact open subsets of $\Spec(\cM(S))$ are closed (see  \cite[Lemma~\href{https://stacks.math.columbia.edu/tag/092F}{092F}\,(2)$\Rightarrow$(3) and Lemma~\href{https://stacks.math.columbia.edu/tag/04MG}{04MG}\,(7)$\Rightarrow$(4)]{SP}), hence that $\Pic(\Spec(\cM(S)))=0$. The Kummer exact sequence $0\to \Z/n\xrightarrow{1\mapsto\exp(\frac{2\pi i}{n})}\G_m\xrightarrow{n}\G_m\to 0$ and Hilbert 90 (see \cite[Exp.\,IX, Th\'eor\`eme 3.3]{SGA43}) therefore yield an isomorphism 
\begin{equation}
\label{Kummerproduct}
\Br(\cM(S))[n]\simeq H^2_{\et}(\Spec(\cM(S)),\Z/n).
\end{equation}

 Use \eqref{Kummerproduct} to view $\alpha$ as an element of $H^2_{\et}(\Spec(\cM(S)),\Z/n)$. Applying \cite[Theorem~\href{https://stacks.math.columbia.edu/tag/09YQ}{09YQ}]{SP} (with $\cF_i=\G_m$) shows the existence of a nonzerodivisor $a\in\cO(S)$ and, setting~$U:=\Spec(\cO(S)[\frac{1}{a}])$, of a class~$\alpha'\in H^2_{\et}(U,\Z/n)$ with ${\alpha'|_{\Spec(\cM(S))}=\alpha}$. Define~$Z:=\{a=0\}\subset S$, so that $U^{\an}=S\setminus Z$. Let ${\beta\in H^2(S\setminus Z,\Z/n)}$ be the image of $\alpha'$ by the comparison morphism 
$$H^2_{\et}(U,\Z/n)\to H^2(S\setminus Z,\Z/n)$$
(see \S\ref{paranalytification}). Let $(D_i)_{i\in I}$ be the irreducible components of $Z$. Let $Z^{\sing}$ be the singular locus of $Z$, set $Z^0:=Z\setminus Z^{\sing}$ and $D_i^0:=D_i\cap Z^0$. Let $\gamma=(\gamma_i)_{i\in I}$ be the element of~$H^1(Z^0,\Z/n)=\prod_{i\in I} H^1(D_i^0,\Z/n)$ image of~$\beta$ in the Gysin long exact sequence
\begin{equation}
\label{Gysin1}
H^2(S\setminus Z^{\sing},\Z/n)\to H^2(S\setminus Z,\Z/n)\to H^1(Z^0,\Z/n).
\end{equation}

\begin{lem}
\label{comparisonresidues}
Fix $i\in I$. The class $\gamma_i\in H^1(D_i^0,\Z/n)$ vanishes.
\end{lem}

\begin{proof}
Let $T_i$ be the connected component of $S$ containing $D_i$. By compatibility of $\res_{D_i}$ with the residue map in \'etale cohomology (see \cite[Theorems~1.4.14 and 2.3.5]{CTS}) and since $\res_{D_i}(\alpha|_{\Spec(\cM(T_i))})=0$, the class $\alpha|_{\Spec(\cM(T_i))}$ lifts to $H^2_{\et}(\Spec(\cO(S)_{D_i}),\Z/n)$. A limit argument based on \cite[Theorem~\href{https://stacks.math.columbia.edu/tag/09YQ}{09YQ}]{SP} shows  that it further lifts to a class~$\alpha'_i\in H^2_{\et}(\Spec(\cO(T_i)[\frac{1}{a_i}]),\Z/n)$ for some nonzerodivisor $a_i\in\cO(T_i)$ that does not vanish identically on~$D_i$. Another application of~\cite[Theorem~\href{https://stacks.math.columbia.edu/tag/09YQ}{09YQ}]{SP} shows that, after possibly changing $a_i$, we may assume that the restrictions of $\alpha'$ and $\alpha'_i$ to $H^2_{\et}(\Spec(\cO(T_i)[\frac{1}{a_ia}]),\Z/n)$ coincide.

Define $Z_i:=\{a_i=0\}\subset T_i$ and let $\beta_i\in H^2(T_i\setminus Z_i,\Z/n)$ be the image of~$\alpha'_i$ by the comparison morphism $H^2_{\et}(\Spec(\cO(T_i)[\frac{1}{a_i}]),\Z/n)\to H^2(T_i\setminus Z_i,\Z/n)$ of \S\ref{paranalytification}. Our choice of $a_i$ implies that the restrictions of $\beta$ and $\beta_i$ to $H^2(T_i\setminus (Z\cup Z_i),\Z/n)$ coincide. Set $W_i:=((Z\cap T_i)\setminus D_i^0)\cup Z_i$. The Gysin long exact sequence 
$$H^2(T_i\setminus W_i,\Z/n)\to H^2(T_i\setminus ((Z\cap T_i)\cup Z_i),\Z/n)\to H^1(D_i^0\setminus (Z_i\cap D_i^0),\Z/n)$$
and its compatibility with \eqref{Gysin1} therefore show that $\gamma_i|_{D_i^0\setminus (Z_i\cap D_i^0)}$ vanishes. As the restriction map $H^1(D_i^0,\Z/n)\to H^1(D_i^0\setminus (Z_i\cap D_i^0),\Z/n)$ is injective (because~$Z_i\cap D_i^0$ is a nowhere dense analytic subset of the complex manifold $D_i^0$), one has $\gamma_i=0$.
\end{proof}

We resume the proof of Theorem \ref{th5+}. Lemma \ref{comparisonresidues} shows that $\gamma=0$, hence that~$\beta\in H^2(S\setminus Z,\Z/n)$ lifts to a cohomology class~$\beta'\in H^2(S\setminus Z^{\sing},\Z/n)$. As~$S$ is a manifold and $Z^{\sing}$ has codimension $\geq 2$ in $S$, the class $\beta'$ further lifts to a class $\beta''\in H^2(S,\Z/n)$ (stratifying $Z^{\sing}$ by its singular locus, the singular locus of its singular locus, etc., one reduces to the case where  $Z^{\sing}$ is smooth in which case this follows from a Gysin long exact sequence).

Let $\partial:H^2(S,\Z/n)\to H^3(S,\Z)[n]$ be the connecting morphism associated to the short exact sequence~${0\to\Z\xrightarrow{n}\Z\to\Z/n\to 0}$. By Theorem \ref{cor1+}, the natural morphism $\Br(\cO(S))\to H^3(S,\Z)_{\tors}$ is an isomorphism. After subtracting from~$\alpha$ the element of $\Br(\cO(S))$ corresponding to $\partial(\beta'')\in H^3(S,\Z)[n]$ and modifying~$\alpha'$,~$\beta$,~$\beta'$ and~$\beta''$ accordingly, we may assume that $\partial(\beta'')=0$. This shows that~$\beta''$ is the reduction modulo $n$ of a class $\delta\in H^2(S,\Z)$.

The exponential exact sequence $0\to \Z\xrightarrow{2\pi i}\cO_S\xrightarrow{\exp}\cO_S^*\to 0$ shows that $\delta$ is the first Chern class of a holomorphic line bundle $\cL$ on $S$. As $S$ is Stein, one can choose~$\sigma\in H^0(S,\cL)$ and $\tau\in H^0(S,\cL^{-1})$ whose zero loci are nowhere dense in $S$. Replacing the nonzerodivisor $a\in\cO(S)$ with $a\sigma\tau\in\cO(S)$, one can ensure that~$\delta|_{S\setminus Z}\in H^2(S\setminus Z,\Z)$ vanishes, hence that $\beta=0$.

At this point, consider the Leray spectral sequence 
\begin{equation}
\label{Leray}
E_2^{p,q}=H^p_{\et}(U,\RR^q(\varepsilon_U)_*\Z/n)\Rightarrow H^{p+q}(S\setminus Z,\Z/n)
\end{equation}
of the morphism of sites $\varepsilon_{U}:(S\setminus Z)_{\cl}=(U^{\an})_{\cl}\to U_{\et}$ (see \S\ref{paranalytification}). Taking into account that the natural morphism $\Z/n\to(\varepsilon_U)_*\Z/n$ is an isomorphism (see \cite[Corollary 3.12]{Steinsurface}), the spectral sequence \eqref{Leray} induces an exact sequence
\begin{equation}
\label{Lerayexact}
H^0_{\et}(U,\RR^1(\varepsilon_U)_*\Z/n)\to H^2_{\et}(U,\Z/n)\to H^2(S\setminus Z,\Z/n).
\end{equation}
As $\beta\in H^2(S\setminus Z,\Z/n)$ vanishes, the class $\alpha'\in H^2_{\et}(U,\Z/n)$ lifts in \eqref{Lerayexact} to a class $\alpha''\in H^0_{\et}(U,\RR^1(\varepsilon_U)_*\Z/n)$. There exists a nonzerodivisor~$b\in\cO(S)$ such that~$\alpha''|_{\Spec(\cO(S)[\frac{1}{ab}])}$ vanishes (unravelling the definition of a higher direct image, this claim follows from \cite[Lemma 6.4]{Steinsurface}). After replacing $a$ with $ab$, we may therefore assume that $\alpha''=0$, hence that $\alpha'=0$, hence a fortiori that $\alpha=0$. The theorem is proved.
\end{proof}

\section{The index for Brauer classes of Stein algebras}

Let $(X,\cO_X)$ be a locally 
%Important, see [Grothendieck, Groupe de Brauer I, Remarque 5.12].
ringed site and fix $\alpha\in \Br(X,\cO_X)$. The index $\ind(\alpha)$ of $\alpha$ is the gcd of the degrees of the Azamuya algebras representing $\alpha$. In view of the first paragraph of the proof of \cite[Proposition~3.1.2.1]{Lieblich}, the index $\ind(\alpha)$ can equivalently be defined as the gcd of the ranks of all locally free $\alpha$-twisted sheaves (in the sense of \S\ref{partwisted}, where one identifies~$\alpha$ and its image in $H^2(X,\cO_X^{\times})$ by \eqref{injBrH2}).

\subsection{Algebraic index versus topological index}
\label{secindex}

\begin{thm}
\label{th10+}
Let $S$ be a finite-dimensional Stein space. For~$\alpha\in \Br(\cO(S))$ with images $\alpha^{\an}\in \Br_{\an}(S)$ and $\alpha^{\topo}\in\Br_{\topo}(S)$, one has $\ind(\alpha)=\ind(\alpha^{\an})=\ind(\alpha^{\topo})$.
\end{thm}

\begin{proof}
Associating to an Azumaya algebra $\cA$ on $\Spec(\cO(S))$ the holomorphic Azumaya algebra $\cA^{\an}$ and to a holomorphic Azumaya algebra $\cB$ on $S$ the topological Azumaya algebra $\cB\otimes_{\cO_S}\cC_S$ shows that~$\ind(\alpha^{\an})\mid\ind(\alpha)$ and $\ind(\alpha^{\topo})\mid\ind(\alpha^{\an})$.

Let $\cA$ be a degree $n$ topological Azumaya algebra on $S$. The proof of Proposition~\ref{propcompa1} shows that one can endow $\cA$ with a structure of degree $n$ holomorphic Azumaya algebra on $S$. As $S$ is finite-dimensional, the proof of Proposition \ref{propcompa2} shows that $\cA$ is the analytification of a degree $n$ Azumaya algebra on $\Spec(\cO(S))$. This proves the converse divisibility $\ind(\alpha)\mid\ind(\alpha^{\topo})$. 
\end{proof}

\begin{cor}
\label{cor:per-ind-bound}
    Let $S$ be an $n$-dimensional Stein space. 
    For any $\alpha \in \Br(\cO(S))$, 
    \begin{equation*}
        \ind(\alpha) \mid (\lceil n/2 \rceil - 1)! \ \per(\alpha)^{\lceil n/2 \rceil - 1}
    \end{equation*}
\end{cor}

\begin{proof}
    By Theorem~\ref{th10+}, it suffices to prove the same bound for $\alpha^{\topo} \in \Br_{\topo}(S)$.
    In fact, the bound in Corollary~\ref{cor:per-ind-bound} holds more generally for a topological Brauer class~$\alpha$ on an $n$-dimensional CW complex $T$ (note that $S$ has the homotopy type of a CW complex of dimension $n$ by \cite[Korollar]{Hamm}).
    Indeed, as in the proof of Theorem~\ref{th2+}, one reduces to the case when $T$ is a finite CW complex of dimension~$n$. One can then apply the main result from \cite{aw3} with $d = \lceil n/2 \rceil$. 
\end{proof}

\begin{rem}
    Antieau and Williams have conjectured in \cite{aw3} that for each $d \geq 1$ and each integer $\ell \geq 2$ prime to $(d - 1)!$, there exists a $2d$-dimensional finite CW complex $T$ and a topological Brauer class $\alpha\in\Br_{\topo}(T)$ of period $\ell$ such that
    \begin{equation*}
        \ind(\alpha) = \per(\alpha)^{d - 1} = \ell^{d - 1}.
    \end{equation*}
    If true, their conjecture would imply (by way of Lemma~\ref{lemStein} and Theorem~\ref{th10+}) that the bound in Corollary~\ref{cor:per-ind-bound} is sharp for Brauer classes of period prime to $(n/2 - 1)!$ on Stein spaces of even dimension $n$.
\end{rem}

\begin{rem}
    Fix an integer $\ell \geq 2$, 
    and let $S$ be a $3$-dimensional Stein manifold with the homotopy type of the Moore space $M(\Z/\ell, 2)$ (using Theorem~\ref{lemStein} and the fact that $M(\Z/\ell, 2)$ is a $3$-dimensional CW complex by construction).
    From Theorem~\ref{th2+} and the universal coefficient theorem, one calculates that $\Br_{\topo}(S) = \Z/\ell$.
    
    This example shows that it is not generally possible to improve in the ceiling~$\lceil n/2 \rceil$ to a floor $\lfloor n/2 \rfloor$ everywhere in Corollary~\ref{cor:per-ind-bound}.
\end{rem}

\subsection{Twisted topological K-theory}
\label{partwistedKth}

The cohomology theory called \textit{twisted topological K-theory} was introduced by Donovan and Karoubi in \cite{DK} and developed by Atiyah and Segal \cite{AS} and others. We collect here the necessary background on twisted topological K-theory used in the proof of Theorem \ref{th9+} in \S\ref{parindOM}.

Let $T$ be a topological space and let $\cA$ be a topological Azumaya algebra of index~$n$ on $T$. 
An $\cA$-\textit{bundle} is a right $\cA$\nobreakdash-module that is locally free of finite rank as a $\cC_T$\nobreakdash-module. We define $K(T,\cA)$ to be the K-theory of the exact category of $\cA$-bundles (see \cite[\S5]{DK}). This definition is not the correct one for infinite-dimensional CW complexes (see \cite{AS}), but it will be sufficient for our purposes.

The rank of an $\cA$-bundle $\cE$ as a $\cC_T$-module is of the form $nr$ for some locally constant function $r:T\to \Z_{\geq 0}$ (reduce to the case~$\cA=M_n(\cC_T)$ where it follows from Morita theory). We call $r$ the $\cA$-\textit{rank} of~$\cE$ (note that $\cA$ has $\cA$-rank $n$). One defines the $\cA$-\textit{rank} of an element of $K(T,\cA)$ by additivity: it is a locally constant function $r:T\to\Z$. If $r$ is constant (\eg if $T$ is connected), we view it as an integer. Let $K(T,\cA)_r\subset K(T,\cA)$ be the subgroup of elements of $\cA$-rank $r$. 

For $0\leq r\leq nN$, we let $\Gr_{(T,\cA)}(r,N)\to T$ be the locally trivial fibration whose 
fiber over $t\in T$ is the space of $\cA_t$-stable subspaces of $\cA_t$-rank $r$ of~$\cA_t^{\oplus N}$.
Since~$\cA$ is locally isomorphic to $M_n(\cC_T)$, it follows from Morita theory that the fibers of~$\Gr_{(T,\cA)}(r,N)\to T$ are homeomorphic to the Grassmannian $\Gr(r,nN)$ of~$r$\nobreakdash-di\-men\-sion\-al subspaces of  $\C^{nN}$. 

Define $\BU_{(T,\cA)}(r):=\varinjlim_N \Gr_{(T,\cA)}(r,N)$ where the colimit is induced by the inclusions $\cA^{\oplus N}\subset \cA^{\oplus N+1}$. Let $\cU_{(T,\cA)}(r)$ be the universal $\cA$-bundle of~$\cA$\nobreakdash-rank~$r$ on~$\BU_{(T,\cA)}(r)$. Set ${\BU_{(T,\cA), r}:=\varinjlim_i \BU_{(T,\cA)}(r+ni)}$ where the colimit is induced by $\cE\mapsto\cE\oplus \cA$. The locally trivial fibrations ${\BU_{(T,\cA)}(r)\to T}$ and ${\BU_{(T,\cA), r}\to T}$ respectively have fibers homeomorphic to the classifying spaces~$\BU(r)$ and~$\BU$.

The following lemma is an analogue in our context of the assertion that $\BU\times\Z$ represents usual topological K-theory (this corresponds to the case $\cA=\cC_T$). As the proof of Lemma \ref{lemtwistedKth} is entirely similar to the proof of this classical assertion (for which see \eg \cite[Chap.\,24, \S1]{Concise}), we omit it.

\begin{lem}
\label{lemtwistedKth}
Let $T$ be a connected compact Hausdorff topological space and let $\cA$ be topological Azumaya algebra of degree $n$ on $T$. Fix $r\in\Z$. There is a bijection
\begin{equation}
 \left\{  \begin{array}{l}
    \textrm{homotopy classes of sections}\\ \hspace{2em}\textrm{of \,} \BU_{(T,\cA), r}\to T
  \end{array}\right\}
    \stackrel[]{}{\isoto} 
K(T,\cA)_r
\end{equation}
associating with $\sigma:T\to \BU_{(T,\cA)}(r+ni)$ the element $[\sigma^*(\cU_{(T,\cA)}(r+ni))]-i[\cA]$.
\end{lem}

\subsection{Index over \texorpdfstring{$\cO(S)$}{O(S)} versus index over \texorpdfstring{$\cM(S)$}{M(S)}}
\label{parindOM}

The following proposition is key to the proof of Theorem \ref{th9+}.

\begin{prop}
\label{proptwistedresolution}
Let $S$ be a connected Stein manifold. Let $\cA$ be a holomorphic Azumaya algebra of degree $n$ on $S$ with class $\alpha\in\Br_{\an}(S)$. Let $\cF$ be a coherent right $\cA$\nobreakdash-module on~$S$ that is generically of rank $rn$ (as an $\cO_S$-module). Then $\ind(\alpha)\mid r$.
\end{prop}

\begin{proof}
Let $d$ be the complex dimension of $S$. Fix $s\in S$.
Let $\rho:S\to\R$ be a $\ci$ strictly plurisubharmonic exhaustion function on $S$ (see \cite[Theorem 5.1.6]{Hormander}). Use Sard's theorem to choose a strictly increasing sequence $(c_i)_{i\geq 1}$ of regular values of $\rho$ tending to infinity. Set $S_i:=\{x\in S\mid \rho(x)\leq c_i\}$. For all $c\in\R$, the complex manifold~$\{x\in S\mid\rho(x)<c\}$ is Stein (apply \cite[Theorem 5.2.10]{Hormander} to $x\mapsto\frac{1}{c-\rho(x)}$). It follows that the compact subset $S_i\subset S$ is Stein in the sense that it admits a basis of Stein open neighborhoods. Using \cite[Theorem 10.6]{Munkres}, one can endow~$S$ with a structure of CW complex such that the $S_i$ are finite subcomplexes. Replacing $S_i$ with its connected component containing $s$, we may assume that $S_i$ is connected.

Fix $i\geq 1$. As $S_i$ is a Stein compact set and $\cF$ is coherent, there is a resolution 
\begin{equation}
\label{resSi}
0\to\cE^i_d\to\dots\to\cE^i_0\to\cF|_{U_i}\to 0
\end{equation}
of $\cF|_{U_i}$, on a Stein open neighborhood $U_i$ of $S_i$, by $\cA|_{U_i}$-bundles (right $\cA|_{U_i}$\nobreakdash-mod\-ules that are locally free of finite rank as $\cO_{U_i}$-modules); one can also choose~$\cE^i_0,\dots,\cE^i_{d-1}$ to be free right $\cA|_{U_i}$\nobreakdash-mod\-ules. Using that~$\cF$ is generated (as an~$\cO_{S}$\nobreakdash-mod\-ule hence also as an~$\cA$-module) by finitely many global sections over some $U_i$, one constructs the morphism~${\cE^i_0\to\cF|_{U_i}}$. Applying the same argument to the kernel of~${\cE^i_0\to\cF|_{U_i}}$ and iterating (possibly shrinking~$U_i$), one constructs~$\cE^i_1,\dots,\cE^i_{d-1}$. The~$\cA|_{U_i}$\nobreakdash-module~$\cE^i_d$ chosen so \eqref{resSi} is exact is a locally free $\cO_{U_i}$-module by Hilbert's syzygy theorem. Shrinking the $U_i$, one can moreover ensure that~${U_i\subset U_{i+1}}$.

 As $U_i$ is Stein, one can lift the identity of $\cF|_{U_i}$ to a morphism ${f^i:\cE^{i+1}_{\bullet}|_{U_i}\to\cE^i_{\bullet}}$ of complexes of~$\cA|_{U_i}$\nobreakdash-bundles (choose first $f_0^i$, then $f^i_1$, etc.). The mapping cone~$C(f^i)_{\bullet}$ of~$f^i$ is then an exact complex of $\cA|_{U_i}$-bundles with~${C(f^i)_j=\cE^{i+1}_{j-1}|_{U_i}\oplus \cE^i_j}$. 
 
Consider the topological Azumaya algebra $\cB:=\cA\otimes_{\cO_S}\cC_S$ on $S$ and let $\cB_i$ be its restriction to $S_i$. The element $\kappa_i:=\sum_{j=0}^d (-1)^j[(\cE^i_j\otimes_{\cO_{U_i}}\cC_{U_i})|_{S_i}]$ of $K(S_i,\cB_i)$ (see~\S\ref{partwistedKth}) has $\cB_i$-rank $r$. The complex $(C(f^i)_{\bullet}\otimes_{\cO_{U_i}}\cC_{U_i})|_{S_i}$ shows that $\kappa_{i+1}|_{S_i}=\kappa_i$.

The class $\kappa_i$ corresponds to a continuous section $\sigma_i:S_i\to\BU_{(S_i,\cB_i),r}$, well defined up to homotopy, of $\BU_{(S,\cB),r}\to S$ over $S_i$ (see Lemma \ref{lemtwistedKth}). As~${\kappa_{i+1}|_{S_i}=\kappa_i}$, the section~$\sigma_{i+1}|_{S_i}$ is homotopic to~$\sigma_i$. For all $i\geq 1$, one can modify $\sigma_{i+1}$ to ensure that~$\sigma_{i+1}|_{S_i}=\sigma_i$ (to see it, note that $\sigma_{i+1}$ and the homotopy yield a section of~$\BU_{(S,\cB)}(r)\times [0,1]\to S\times [0,1]$ over $S_{i+1}\times\{0\}\cup S_i\times [0,1]$, which one can extend to $S_{i+1}\times [0,1]$ because the inclusion $S_{i+1}\times\{0\}\cup S_i\times [0,1]\subset S_{i+1}\times [0,1]$ is a homotopy equivalence). 
Doing so for $i=1$, then $i=2$, etc., ensures that the $\sigma_i$ are compatible. This yields a section $\sigma:S\to \BU_{(S,\cB),r}$ of~$\BU_{(S,\cB),r}\to S$ over $S$.

Choose $i\geq 0$ such that $r+ni\geq d$. Then the homotopy fiber $F$ of the inclusion~$\BU(r+ni)\to\BU$ has the property that $\pi_k(F)=0$ for $0<k\leq 2d$ (this follows from the Hurewicz theorem in the form \cite[Theorem 4.27]{Hatcher}). Consider the problem of finding a section $\tau$ of $\BU_{(S,\cB)}(r+ni)\to S$ that is homotopic to~$\sigma$ as a section of $\BU_{(S,\cB),r}\to S$. Obstruction theory provides successive obstructions to this problem, that live in $H^{k+1}(S,\underline{\pi}_k(F))$, where~$\underline{\pi}_k(F)$ is a local system on $S$ with fiber~$\pi_k(F)$ (see \eg \cite[Theorem 34.2]{Steenrod}). These groups vanish if $k\leq 2d$ because~$\pi_k(F)=0$ and if $k\geq 2d$ because $S$ is a $2d$-dimensional~$\ci$ manifold. It follows that such a section~$\tau:S\to \BU_{(S,\cB)}(r+ni)$ exists (see \cite[Corollary~34.4]{Steenrod}).

The $\cB$-bundle $\tau^*(\cU_{(S,\cB)}(r+ni))$ on $S$ has $\cB$-rank $r+ni$ (in the notation of~\S\ref{partwistedKth}). Let $\beta\in\Br_{\topo}(S)$ be the class of $\cB$. By \cite[Theorem 1.3.7]{Caldararu}, there exists a~$\beta$\nobreakdash-twisted sheaf of rank $r+ni$ on $S$, so that $\ind(\beta)\mid r+ni$. As clearly $\ind(\beta)\mid n$, it follows that $\ind(\beta)\mid r$. The proof of Theorem \ref{th10+} now shows that~$\ind(\alpha)\mid r$.
\end{proof}

\begin{thm}
\label{th9+}
Let $S$ be a connected Stein manifold. Fix $\alpha\in \Br(\cO(S))$ and let~$\alpha_{\cM(S)}$ be its image in $\Br(\cM(S))$. Then $\ind(\alpha)=\ind(\alpha_{\cM(S)})$.
\end{thm}

\begin{proof}
It is obvious that $\ind(\alpha_{\cM(S)})\mid \ind(\alpha)$ and we prove the converse divisibility. 

Let $\pi:\kX\to\Spec(\cO(S))$ be a $\G_m$-gerbe representing~$\alpha$ (see \S\ref{partwisted}). Let~$\cG$ be a locally free~$\alpha_{\cM(S)}$\nobreakdash-twisted sheaf of rank $r$. Extend $\cG$ to a quasi-coherent sheaf of finite presentation $\cH$ on~$\kX$ (combine \cite[Lemma~4.3~(C2)$\Rightarrow$(E2)]{Rydh1} and \cite[Corollary B]{Rydh2}). Replacing~$\cH$ with one summand of the decomposition \cite[Proposition~2.2.1.6]{Lieblichmoduli}, we may ensure that $\cH$ is an $\alpha$-twisted sheaf on $\Spec(\cO(S))$.

Let $\cA$ be a degree $n$ Azumaya algebra on $\Spec(\cO(S))$ representing $\alpha$. Using \cite[Theorem 1.3.7]{Caldararu},
%See also [van Bree, Virasoro constraints and moduli of twisted sheaves, Lemma 3.1.44].
one sees that there exists a finitely presented quasi-coherent sheaf $\cF$ on $\Spec(\cO(S))$ with a structure of right $\cA$-module such that $\cF_{\cM(S)}$ has rank $rn$ (as a module over $\cM(S)$).
Let $\cA^{\an}$ and $\cF^{\an}$ be the analytifications of $\cA$ and $\cF$ in the sense of \S\ref{paranalytification}, so~$\cA^{\an}$ is a holomorphic Azumaya algebra of degree~$n$ on $S$ and $\cF^{\an}$ is a coherent right~$\cA^{\an}$\nobreakdash-module on $S$ that is generically of rank~$rn$ (as an $\cO_S$-module). It follows from Proposition \ref{proptwistedresolution} that $\ind(\alpha^{\an})\mid r$, hence, by Theorem \ref{th10+}, that~${\ind(\alpha)\mid r}$. This completes the proof of the theorem.
\end{proof}

\subsection{Period versus index}

\begin{thm}
\label{th8+}
There exist a connected Stein manifold $S$ and a class $\alpha\in\Br(\cM(S))$ in the image of
$\Br(\cO(S))\to\Br(\cM(S))$ such that $\ind(\alpha)\neq\per(\alpha)$.
\end{thm}

\begin{proof}
Let $T$ be a connected finite CW complex carrying a topological Brauer class whose period and index differ (say, period $2$ and index $8$ as in \cite[Theorem~B]{AW2}). Use Theorem \ref{lemStein} to find a connected Stein manifold $S$ with the homotopy type of~$T$, so there exists~$\gamma\in \Br_{\topo}(S)[2]$ with $\ind(\gamma)=8$. Let $\beta\in\Br(\cO(S))[2]$ be such that~$\beta^{\topo}=\gamma$ (see Theorem \ref{cor1+}). Let $\alpha\in \Br(\cM(S))[2]$ be the image of~$\beta$. Applying Theorems \ref{th9+} and \ref{th10+} shows that ${\ind(\alpha)=\ind(\beta)=\ind(\gamma)=8}$.
\end{proof}

\bibliographystyle{myamsalpha}
\bibliography{BrauerStein}

\end{document}